\documentclass[12pt]{amsart}
\usepackage[T1]{fontenc}
\usepackage{graphicx} 
\usepackage{amsfonts,amsthm,latexsym,amsmath,amssymb,amscd,amsmath, mathrsfs, epsf, tikz-cd, enumitem}
\usepackage{hyperref}
\usepackage[left=25mm, right= 25mm, top=20mm, bottom=20mm, includefoot, includehead]{geometry}
\usepackage[textsize=tiny]{todonotes}
\usepackage{url}

\usepackage{tikz}
\usetikzlibrary{decorations.text,calc,arrows.meta}

\newtheorem{theorem}{Theorem}[section]
\newtheorem*{theo*}{Theorem}
\newtheorem{corollary}[theorem]{Corollary}
\newtheorem{lemma}[theorem]{Lemma}
\newtheorem{proposition}[theorem]{Proposition}

\newtheorem{claim}{Claim}
\newtheorem{ex*}{Example}[section]
\theoremstyle{definition}

\theoremstyle{remark}
\newtheorem{remark}[theorem]{Remark}
\theoremstyle{definition}

\newcommand{\brk}{{\underline{{\bf rk}}}}

\newcommand{\vv}{\mathbf{v}}
\newcommand{\ee}{\mathbf{e}}

\newcommand{\oo}{\mathbf{0}}
\newcommand{\Hom}{\mathrm{Hom}}
\newcommand{\LL}{\mathcal{L}}

\newcommand{\sat}{\mathrm{sat}}

\newcommand{\HF}[2]{H_{#1}(#2)}
\newcommand{\Slip}[2]{\mathrm{Slip}_{#1,#2}}

\newcommand{\CC}{\mathbb{C}}
\newcommand{\OO}{\mathcal{O}}

\newcommand{\PP}{\mathbb{P}}

\newcommand{\NN}{\mathbb{N}}
\newcommand{\ZZ}{\mathbb{Z}}

\newcommand{\Ann}{\mathrm{Ann}}
\newcommand{\Hilb}{\mathrm{Hilb}}

\makeatletter
\@namedef{subjclassname@2020}{%
	\textup{2020} Mathematics Subject Classification}
\makeatother

\date{}

\subjclass[2020]{}

\title{Border rank lower bounds beyond weak border apolarity}

\date{}

\author{Tomasz Ma\'ndziuk}
\address{Texas A\&M University, Department of Mathematics,
	College Station, TX 77843-3368, USA}
\email{t.mandziuk@tamu.edu}

\begin{document}
	
	\begin{abstract}
    We provide first examples of border rank lower bounds obtained using border apolarity beyond   \lq\lq weak border apolarity\rq\rq. 
    Border apolarity introduced by Buczy\'nska-Buczy\'nski characterizes tensors of border rank greater than $r$ as those whose annihilator does not contain an ideal with certain properties depending on $r$ and the ambient space of the tensor. Previous applications of border apolarity depended on the algorithm developed by Conner-Harper-Landsberg. For a tensor with a large symmetry group it lists a set of ideals that satisfy some of the properties mentioned above. We give first examples where the list is non-empty but we are able to prove that no ideal on this list satisfies all necessary conditions coming from border apolarity.
    The new lower bound examples include tensors of interest in complexity:
the first Sch\"onhage triple sum of matrix multiplication tensors where the border
rank was unknown,  and the direct sum of two copies of a tensor that could potentially be used
to prove the exponent of matrix multiplication is two.  A third  example of
interest in  group theory related to work of A. Leitner is also analyzed. \end{abstract}

    \maketitle

\section{Introduction}
The goal of this paper is to illustrate, on several tensors of interest, new methods for establishing border rank lower bounds based on border apolarity developed by Buczy\'nska and Buczy\'nski in \cite{BB19}. 
The border rank of a tensor is a standard measure of the complexity of computing the corresponding bilinear map,
see \cite[Chaps. 14-15]{BCS}.
It is a famously difficult problem to determine the border rank of explicit tensors. Especially important are matrix multiplication
tensors as determining the complexity of matrix multiplication is one of the most important problems in algebraic complexity theory.
 For example, for the $2\times 2$-matrix multiplication tensor in $\CC^4\otimes \CC^4 \otimes \CC^4$, the upper bound of $7$ was known since the work of Strassen in 1969 \cite{Str69}. On the other hand, the matching lower bound was found by Landsberg in 2006 \cite{Lan06} and a first hand-checkable proof appeared in \cite{chl19}. The latter proof is based on border apolarity. The most commonly used techniques for lower bounding the border rank of tensors predating border apolarity are Koszul flattenings introduced in \cite{LO15} as a special case of Young flattenings \cite{LO11}. These extend Strassen's equations \cite{Str83}.

In its full strength, border apolarity characterizes tensors of border rank at most $r$. A tensor $T$ has border rank at most $r$ if and only if there exists an ideal $I\subseteq \Ann(T)$ that satisfies properties of two types: (i) numerical, and (ii) topological. See Section~\ref{s:notation} for the precise statements of (i) and (ii). So far, border apolarity was mainly used to improve the lower bounds on border ranks. This requires one to prove that for given $T$ and $r$ there is no ideal $I\subseteq \Ann(T)$ that satisfies properties (i) and (ii).
For a given ideal, checking if it satisfies the numerical condition is algorithmic but checking the topological condition is at this writing more of an art, and quite difficult to implement. Therefore, in practice one usually starts with the weaker version of border apolarity called weak border apolarity that says that if the border rank of $T$ is at most $r$, then there is an ideal $I\subseteq \Ann(T)$ that satisfies the numerical condition (i). So far all applications of border apolarity for border rank lower bounds were based on this weaker version. See for example \cite{Bar22}, \cite{chl19}, \cite{chl23}, \cite{Fla23}. For an example of applying border apolarity to get an upper bound on the border rank see \cite{KLSS}.

\subsection{Results}
In this paper we go beyond weak border apolarity to determine the border rank of three
tensors of interest. First some background:

In Sch\"onhage's breakthrough paper \cite{Sch81}, in addition to proving his upper bound of $2.55$
for the exponent of matrix multiplication using the direct sum of two matrix multiplication tensors, he obtained
a similar bound using the direct sum of three such. In the case of the sum of the two, he used
  the actual border rank but for the sum of three, he just used an upper bound estimate, so there
 is potential room for improvement.

Given positive integers $k,l,m$, let $M_{\langle k,l,m\rangle}$ denote the matrix multiplication tensor corresponding to the bilinear map
\[
\mathrm{Mat}_{k\times l} \times \mathrm{Mat}_{l\times m} \to \mathrm{Mat}_{k\times m}.
\]
It is a tensor in $\CC^{kl}\otimes \CC^{lm}\otimes \CC^{km}$. We write $M_{\langle k\rangle}$ instead of $M_{\langle k,k,k\rangle}$.
In \cite[Section~8]{Sch81} Sch\"onhage considers the following tensor 
\begin{equation}\label{eq:SchonhageTensor}
T=M_{\langle 1, u, 2v\rangle}  \oplus M_{\langle v,2,u\rangle} \oplus M_{\langle 2u,v,1\rangle} \in \CC^{2uv+2v+u}\otimes \CC^{2uv+2u+v}\otimes \CC^{uv+2u+2v}
\end{equation}
 and shows that $\brk(T) \le 2uv+2v+4u+4$.
Assuming $v\geq u$, the minimal possible border rank would be $2uv+2v+u$. This still leads to the possibility of improving the bound on the exponent of matrix multiplication.

Weak border apolarity in the case $(u,v)=(1,1)$ proves the border rank is six, 
in the next case $(u,v)=(1,2)$ weak border apolarity gives a lower bound of $11$ which
does not match the known upper bound of $12$.
Our first result is:
\begin{theorem}[Theorem~\ref{thm:Scho}] $\brk(M_{\langle 1, 1, 4\rangle}  \oplus M_{\langle 2,2,1\rangle} \oplus M_{\langle 2 ,2,1\rangle} )=12.$
\end{theorem}

In \cite{CGLV22} it was observed that the unique up to scale skew-symmetric tensor in $(\CC^3)^{\otimes 3}$, denoted
$T_{\mathrm{skewcw},2}$ could potentially prove the exponent of matrix multiplication is two as it has the same value as the smallest small
Coppersmith-Winograd tensor. However its cost (border rank) of five led to initial pessimism.
In \cite{chl19} it was shown that its Kronecker square had a dramatic drop in border rank from
the na\"\i ve $25$ to $17$, and in \cite{HL26} it was observed that one gets a similar drop with 
$T_{\mathrm{skewcw},2}\boxtimes W$, which has border rank $9$ instead of the expected $10$.
Here 
\begin{equation}\label{eq:Wtensor}
W = a_1\otimes b_1\otimes c_2 + a_1\otimes b_2\otimes c_1 + a_2\otimes b_1\otimes c_1
\end{equation}
is the unique up to change of bases rank $3$ tensor in $\CC^2\otimes \CC^2\otimes \CC^2$. 

This led to the question as to what other drops can one get with this tensor.
The next natural thing to try would be 
$T_{\mathrm{skewcw},2}\boxtimes M_{\langle 1\rangle}^{\oplus 2}= T_{\mathrm{skewcw},2}\oplus T_{\mathrm{skewcw},2}$.
Here weak border apolarity says that a drop to $9$ still might be possible. Our next result
is that this is not the case:

\begin{theorem}[Theorem~\ref{thm:skew}] $\brk(T_{\mathrm{skewcw},2}\boxtimes M_{\langle 1\rangle}^{\oplus 2})=10$.
\end{theorem}

Our third computation deals with the following tensor from \cite[Definition~17]{Lei16}
\begin{align}\label{eq:Leitner}
T_{\mathrm{Leit},{5}} = &a_1 \otimes (b_1 \otimes c_1 + b_2 \otimes  c_2 + b_3 \otimes c_3 + b_4 \otimes  c_4 + b_5\otimes c_5)\\
& + a_2 \otimes (b_1 \otimes c_3 + b_3 \otimes c_5) + a_3 \otimes b_1 \otimes c_4 + a_4 \otimes b_2 \otimes c_4 + a_5 \otimes b_2 \otimes c_5.\notag
\end{align}

In the language of \cite{LM17}, it is an example of an $A$-abelian tensor that is not of minimal border rank as it fails the $\mathrm{End}$-closed condition \cite{LM08}.
Weak border apolarity and Koszul flattenings could not rule out $\brk(T_{\mathrm{Leit},{5}}\boxtimes W)$ being $11$ rather than the na\"\i ve $12$. Our third lower bound is
\begin{theorem}[Theorem~\ref{thm:leit}]
$\brk(T_{\mathrm{Leit},{5}}\boxtimes W) = 12$.
\end{theorem}

We also recover the current best lower bound for the border rank of $M_{\langle 3\rangle}$ from \cite{chl19}.
The original proof is based on analyzing $512$ ideals and proving that they all fail the numerical condition.
Here we consider $8$ ideals and prove that they all fail the topological condition. We present a hand-checkable proof for one of the $8$ ideals.

\subsection{Overview of the method}
A key to implementing border apolarity is that when a tensor $T$ has a nontrivial connected solvable group that stabilizes it
   then one needs to consider only ideals satisfying conditions (i) and (ii)
   fixed by the solvable group. Suppose we are given $T$, a connected solvable group $\mathbb{B}$ stabilizing $T$,  and $r$ and we want to prove that the border rank of $T$ is larger than $r$ using border apolarity. We may proceed in two steps. First, we enumerate all $\mathbb{B}$-fixed ideals $I\subseteq \Ann(T)$ that satisfy the numerical condition (i). If there are no such ideals we get the desired bound on the border rank by weak border apolarity. If there are ideals satisfying the numerical condition we need to prove that none of them satisfies the topological condition.
The first step has an algorithmic solution \cite{chl19} which means that currently the main obstacle to successfully applying border apolarity is that we do not have an easy way to test the topological condition. 

In this paper we provide some tests for the topological condition based on earlier results in \cite{Man22} and \cite{Man23} and illustrate their utility by  proving the above optimal lower bounds for the border rank. These are the first examples where the border rank can be computed using border apolarity even though its weak version is not sufficient. Despite the fact that there is a need for developing better tests of the topological condition, we demonstrate that proving that an ideal does not satisfy this condition is not hopeless and border apolarity can still be successfully employed. For other implementations of the topological condition see \cite{JM25}. For the theory of unrestrictions that can reduce the question to testing the topological condition in the minimal border rank case see \cite{JJ26}.

\subsection{Description of the necessary conditions beyond weak border apolarity used in this paper}

We provide two tests of the topological condition. Both are  based on semicontinuity. The first test (Proposition~\ref{prop:ts_test}) consists of computing the dimension of the tangent space to a certain parameter space. As explained in Section~\ref{s:notation}, there is a parameter space for the ideals satisfying the numerical condition (i). We consider a morphism from this parameter space to another scheme and claim that the dimension of the tangent space at the image of $[I]$ must be large enough if $[I]$ satisfies the topological condition (ii). Based on this test, we establish in Theorem~\ref{thm:Scho} the border rank of the simplest case of the Sch\"onhage triple sum for which the previously employed techniques are not sufficient.

The second test of the topological condition (Proposition~\ref{prop:st_general}) is based on the semicontinuity of the dimensions of the fibers of a coherent sheaf. Namely, we consider a coherent sheaf on the parameter space discussed above and prove that if $I$ satisfies the topological condition (ii), then the dimension of the fiber of that sheaf over $[I]$ is large enough. Based on this test we establish in Theorem~\ref{thm:leit} and Theorem~\ref{thm:skew} the border ranks of two tensors. Both tensors are constructed from simpler tensors. The first is an example of a Kronecker product and the second is an example of a direct sum. The border rank is submultiplicative under Kronecker product and subadditive under direct sum. There are examples of border rank being strictly submultiplicative \cite{CGJ19} or subadditive \cite{Sch81}. We prove that for these two tensors there is no border rank drop.

We stress that both tests amount to   computing the dimension of a vector space constructed from $I$ in an explicit way, which makes it easy to test them using a computer. In principle, they can also be performed without a computer which we illustrate for the second test on two examples (see the proof of Theorem~\ref{thm:skew} and Section~\ref{s:m3}). Furthermore, even though enumerating ideals satisfying the numerical condition (i) has an algorithmic solution, it is important to be able to rule out the ideals by looking only at low degree parts of these ideals. This is because, once there is at least one ideal satisfying (i) there are
usually many such ideals. Both tests require only information about a small part of the ideal. In fact, sometimes it is enough to look at only a single degree. This is illustrated  in Section~\ref{s:m3} where we recover the current best lower bound for the border rank of the $3\times 3$ matrix multiplication tensor from \cite{chl19} using our test of the topological condition.

\section{Background and notation}\label{s:notation}
In this section we introduce the notation used in the subsequent sections and recall some background material.

\subsection{Notation}\label{subsec:not}
We work over the field of complex numbers $\CC$. For a complex vector space $V$, its dual is denoted $V^\vee$. Given a nonzero vector $F\in V$ let $[F]$ denote its image in $\PP V$.

For a $\mathbb{Z}^3$-graded module $M$ over a $\mathbb{Z}^3$-graded ring and $\mathbf{u}\in \mathbb{Z}^3$ let $M_\mathbf{u}$ denote the space of homogeneous elements of $M$ of degree $\mathbf{u}$ and $H_{M}\colon \mathbb{Z}^3\to \mathbb{N}$ be the Hilbert function of $M$, i.e., $\mathbf{u}\mapsto \dim_\CC M_{\mathbf{u}}$.

Let $A,B,C$ be positive-dimensional complex vector spaces of dimensions $n_1+1,n_2+1, n_3+1$, respectively.
Let $a_0,\ldots, a_{n_1}$, $b_0,\ldots, b_{n_2}$ and $c_0,\ldots, c_{n_3}$ be bases for $A,B$ and $C$, respectively. Let $\alpha_{0}, \ldots, \alpha_{n_1}, \beta_0,\ldots, \beta_{n_2}$ 
and $\gamma_0,\ldots, \gamma_{n_3}$ be the dual basis for $A^\vee, B^\vee$ and $C^\vee$, respectively.

Let $X=\PP A\times \PP B\times \PP C$ and 
\[
S[X] = \CC[\alpha_{0},\ldots, \alpha_{n_1}, \beta_0,\ldots, \beta_{n_2}, \gamma_0,\ldots, \gamma_{n_3}]
\]
be its Cox ring, the $\mathbb{Z}^3$-graded polynomial ring with $\deg(\alpha_{i}) =(1,0,0)$, $\deg(\beta_j) = (0,1,0)$ and $\deg(\gamma_k) = (0,0,1)$ for all $1\le i\le n_1$, $0\le j\le n_2$ and $0\le k\le n_3$. Given a $\mathbb{Z}^3$-graded subspace $V\subseteq S[X]$, $(V)$ is the ideal of $S[X]$ generated by $V$.  For a $\mathbb{Z}^3$-graded $S[X]$-module $M$ and $\mathbf{u} = (u_1, u_2, u_3)$ with $u_i < 10$ for all $i$ we write $M_{u_1u_2u_3}$ instead of $M_{(u_1,u_2,u_3)}$.

In what follows when  we consider a tensor $T$, if we do not specify otherwise, it is assumed that it lives in $A\otimes B\otimes C$.

\subsection{Macaulay's inverse systems}
Let $\widetilde{S[X]} = \CC_{\mathrm{dp}}[a_{0}, \ldots, a_{n_1}, b_{0},\ldots, b_{n_2}, c_0, \ldots,c_{n_3}]$
be the dual $\mathbb{Z}^3$-graded space of $S[X]$, i.e., $\widetilde{S[X]}_\mathbf{u} = (S[X]_\mathbf{u})^\vee$ for all $\mathbf{u}\in \mathbb{Z}^3$. 
It is an $S[X]$-module with the module structure defined on a monomial $m$ of $\widetilde{S[X]}$ and variable $\theta\in \{\alpha_0,\ldots, \gamma_{n_3}\}$ by 
\[
\theta \lrcorner m = \begin{cases}
 0 & \textrm{if } \theta\textrm{ does not divide } m\\
 m/\theta & \textrm{otherwise.}
\end{cases}
\]
Given a $\mathbb{Z}^3$-homogeneous ideal $I\subseteq S[X]$, let $I^\perp = \{F\in \widetilde{S[X]} \mid \Theta \lrcorner F = 0 \text{ for all } \Theta\in I\}$. It is a $\mathbb{Z}^3$-graded $S[X]$-module. Conversely, given a $\mathbb{Z}^3$-homogeneous $S[X]$-submodule $M\subseteq \widetilde{S[X]}$, let $M^\perp = \{\Theta\in S[X] \mid \Theta\lrcorner F = 0 \text{ for all } F\in M\}$. It is a $\mathbb{Z}^3$-homogeneous ideal of $S[X]$.
In this way we obtain an inclusion reversing bijection 
\[
\big\{\mathbb{Z}^3\textrm{-homogeneous ideals }I\subseteq S[X]\big\} \overset{()^\perp}{\rightleftarrows} \big\{\mathbb{Z}^3\textrm{-homogeneous }S[X]\textrm{-submodules of }\widetilde{S[X]}\big\}.
\]
Given a $\mathbb{Z}^3$-graded ideal $I\subseteq S[X]$ and $\mathbf{u}\in \mathbb{Z}^3$ we have $((I_{\mathbf{u}})^\perp)_\mathbf{u} = (I^\perp)_{\mathbf{u}}$ so we use the simpler notation $I_\mathbf{u}^\perp$ as it should not lead to confusion. For any $\mathbf{u}\in \mathbb{Z}^3$ and a $\mathbb{Z}^3$-homogeneous ideal $I\subseteq S[X]$ we have $\dim_\CC I^\perp_\mathbf{u}  = \HF{S[X]/I}{\mathbf{u}}$.
For a homogeneous element $T\in \widetilde{S[X]}$, we denote the set of all $\Theta\in S[X]$ with $\Theta\lrcorner T = 0$ by $\Ann(T)$. This is a $\mathbb{Z}^3$-homogeneous ideal of $S[X]$ called the \emph{annihilator} of $T$.
\subsection{Border rank and border apolarity}

Let $B_X = (S[X]_{111})$ be the irrelevant ideal. For a $\mathbb{Z}^3$-homogeneous ideal $I\subseteq S[X]$, its saturation with respect to $B_X$ is denoted $I^\sat$.

Given a non-degenerate projective variety $X\subseteq \PP V$ and a point $p\in \PP V$, the $X$-\emph{border rank} of $p$ denoted $\brk_X(p)$ is the smallest integer $r$ such that
$p$ belongs to the Zariski closure of the following set
\[
\{q\in \PP V \mid \textrm{there exist } x_1,\ldots, x_r \in X \textrm{ such that } q\in \langle x_1,\ldots, x_r\rangle\}
\]
where $\langle x_1,\ldots, x_r\rangle$ is the projective linear span of $x_1,\ldots, x_r$. For $F\in V\setminus\{0\}$ we define $\brk_X(F) = \brk_X([F])$.
Here we are mainly interested in the case where $X=\mathbb{P}A \times \mathbb{P}B\times \mathbb{P}C$ with the Segre embedding into $\PP (A\otimes B\otimes C)$ so we suppress $X$ from the notation and write $\brk(F)$.

We now present the main tool of this paper---border apolarity from \cite{BB19}. We state it only for three factor Segre varieties. A slightly more general setup is used in the last section. Let $X$ be the Segre variety as in Subsection~\ref{subsec:not}.

Given a positive integer $r$, let $h_{r,X}\colon \mathbb{Z}^3\to \mathbb{N}$ given by $h_{r,X}(\mathbf{u}) = \min \{r, \dim_\CC S[X]_\mathbf{u}\}$ be the Hilbert function of $r$ general points in $X$.
Given any $h\colon \mathbb{Z}^3\to \mathbb{N}$ there is a projective scheme $\Hilb_{S[X]}^{h}$ called the \emph{multigraded Hilbert scheme} that parameterizes all $\mathbb{Z}^3$-homogeneous ideals $I$ of $S[X]$ with $H_{S/I} = h$ (see \cite{HS04}). The closed point of $\Hilb_{S[X]}^h$ corresponding to ideal $I$ is denoted $[I]$.
In the special case that $h=h_{r,X}$, let $\mathrm{Slip}_{r,X}$ be the Zariski closure of the following set
\[
\{[I]\in \Hilb_{S[X]}^{h_{r,X}} \mid I \textrm{ is radical and }B_X\textrm{-saturated}\}.
\]
This is an irreducible component of $\Hilb_{S[X]}^{h_{r,X}}$ \cite[Proposition~3.13]{BB19} called Slip (\emph{scheme of limits of ideals of points}).

\begin{theorem}[border apolarity~{\cite[Theorems~3.15~and~4.3]{BB19}}]\label{ba}
  Let $X = \PP A\times \PP B \times \PP C \subseteq \PP (A\otimes B\otimes C) = \PP (\widetilde{S[X]}_{111})$ be the Segre variety. If $T\in A\otimes B\otimes C$ is nonzero, then $\brk(T) \le r$ if and only if there exists an ideal $I\subseteq \Ann(T)$ such that $[I]\in \Slip{r}{X}$. 

  Furthermore, if $\mathbb{B}\subseteq \mathrm{Aut}(X)$ is a connected solvable subgroup preserving $[T] \in \PP(\widetilde{S[X]}_{111})$, then $\brk(T) \le r$ if and only if there exists a $\mathbb{B}$-fixed ideal $I\subseteq \Ann(T)$ such that $[I]\in \Slip{r}{X}$.
\end{theorem}

Theorem~\ref{ba} characterizes $[T]$ with $\brk(T) \le r$ but proving that a given $[I]\in \Hilb_{S[X]}^{h_{r,X}}$ is not in $\Slip{r}{X}$ is extremely difficult in general. Therefore, the following necessary condition that is significantly easier to check is often very useful.

\begin{theorem}[weak border apolarity~{\cite[Theorems~1.2]{BB19}}]\label{wba}
 Let $X = \PP A\times \PP B \times \PP C \subseteq \PP (A\otimes B\otimes C) = \PP (\widetilde{S[X]}_{111})$ be the Segre variety. If $T\in A\otimes B\otimes C$ is nonzero and $\brk(T) \le r$, then there exists an ideal $I\subseteq \Ann(T)$ such that $H_{S[X]/I}=h_{r,X}$. 

  Furthermore, if $\mathbb{B}\subseteq \mathrm{Aut}(X)$ is a connected solvable subgroup preserving $[T] \in \PP(\widetilde{S[X]}_{111})$ and $\brk(T) \le r$, then there exists a $\mathbb{B}$-fixed ideal $I\subseteq \Ann(T)$ such that $H_{S[X]/I}=h_{r,X}$.
\end{theorem}

For a given tensor $T$ and positive integer $r$, a $\mathbb{Z}^3$-homogeneous ideal $I\subseteq\Ann(T)$ satisfies the numerical condition (i) alluded to in the introduction if and only if its Hilbert function is $h_{r,X}$, i.e., $[I]$ is a point if $\Hilb_{S[X]}^{h_{r,X}}$. It satisfies the topological condition (ii) if and only if it belongs to $\mathrm{Slip}_{r,X}$.

\subsection{Hilbert function tests}\label{sub:HFtests}
We recall the terminology of Hilbert function tests introduced in \cite{chl19}. Let $T\in A\otimes B\otimes C$ be a concise tensor, $X=\PP A\times \PP B\times \PP C$ and $r$ be a positive integer.

Given $I_{110}\subseteq S[X]_{110}$ of codimension $r$ contained in $\Ann(T)_{110}$ the \emph{$(210)$-map} is the multiplication map
$I_{110}\otimes A^\vee \to S^2A^\vee\otimes B^\vee = S[X]_{210}$. The \emph{$(210)$ test is passed} if the codimension of the image of the $(210)$-map is at least $r$.
This is a necessary condition for the existence of $I$ with Hilbert function $h_{r,X}$ with this choice of $110$ space.

Given $I_{110}\subseteq S[X]_{110}, I_{101}\subseteq S[X]_{101}, I_{011}\subseteq S[X]_{011}$ each of codimension $r$ the \emph{$(111)$-map} is the multiplication map
\[
(I_{011}\otimes A^\vee) \oplus (I_{101}\otimes B^\vee) \oplus (I_{110}\otimes C^\vee) \to A^\vee\otimes B^\vee\otimes C^\vee = S[X]_{111}.
\]
The \emph{$(111)$ test is passed} if the codimension of the image of the $(111)$-map is at least $r$.
This is a necessary condition for the existence of $I$ with Hilbert function $h_{r,X}$ with the given $011, 101$ and $110$ spaces.
A triple $(I_{110}, I_{101}, I_{011})$ passing the $(111)$ test and all $6$ variants of the $(210)$ tests obtained by permuting $A,B$ and $C$ is called a \emph{candidate triple}.

Given a candidate triple, by adding $I_{200}\subseteq S[X]_{200}$ of codimension $r$ we can strengthen the $(210)$ test. Namely, the dimension of the cokernel of the multiplication map
\begin{equation}\label{q:FULL210}
I_{110}\otimes A^\vee \oplus I_{200}\otimes B^\vee \to S^2A^\vee \otimes B^\vee = S[X]_{210}
\end{equation}
needs to be at least $r$. Given a candidate triple $(I_{110}, I_{101}, I_{011})$ and a triple $(I_{200}, I_{020},I_{002})$  where each $I_{\mathbf{u}} \subseteq S[X]_\mathbf{u}$ with $\mathbf{u}\in \{(2,0,0), (0,2,0), (0,0,2)\}$ is either trivial or of codimension $r$, we say that \emph{all total degree $3$ tests are passed} if the cokernels of the six maps obtained from \eqref{q:FULL210} by permuting $A,B,C$ are all of dimension at least $r$. If $I_{200} = 0$, then \eqref{q:FULL210} is just the $(210)$-map so this map has no extra information. However, typically the ideals are extended one degree at a time so it is convenient to initially allow some of the $I_{200}, I_{020},I_{002}$ to be zero as this might be enough to get an elimination.

\section{Tangent space test}
In this section we discuss a special case of the tangent space test presented in \cite[Proposition~5.2]{Man23}. We focus on the case of $X=\mathbb{P}^{n_1}\times \mathbb{P}^{n_2}\times \mathbb{P}^{n_3}$ and 
\begin{equation}\label{eq:ass_for_ts}
r \leq \min \{(n_1+1)(n_2+1), (n_1+1)(n_3+1), (n_2+1)(n_3+1)\}.    
\end{equation}
For the convenience of the reader, we present a short proof in this special case. Note that \eqref{eq:ass_for_ts} is not restrictive as the right hand side is an upper bound on the border rank of any tensor in $\CC^{n_1+1}\otimes \CC^{n_2+1}\otimes \CC^{n_3+1}$.

The following lemma concerns well-known properties of Hilbert functions of saturated ideals of $S[X]$. Note that by the definition of $h_{r,X}$, every ideal in $\Hilb_{S[X]}^{h_{r,X}}$ trivially satisfies (a) and (b) below. The lemma is known to hold when $J$ is an ideal of distinct points \cite[Proposition~3.5]{Tuy02}. The $i$-th standard basis vector of $\mathbb{Z}^3$ is denoted $\mathbf{e}_i$. 

\begin{lemma}{\label{lem:properties_of_HF}}
    Let $X = \PP^{n_1}\times \PP^{n_2} \times \mathbb{P}^{n_3}$ and let $J$ be a $B_X$-saturated ideal in $S[X]$.
    \begin{enumerate}[label = (\alph*)]
        \item  $\HF{S[X]/J}{\vv + \ee_i} \geq  \HF{S[X]/J}{\vv}$ for all $1\leq i \leq 3$ and all $\vv\in \ZZ^3$. 
        \item If $\HF{S[X]/J}{\vv+ \ee_i} =  \HF{S[X]/J}{\vv}$ for some $1\leq i \leq 3$ and $\vv\in \ZZ^3$, then $\HF{S[X]/J}{\vv + 2\ee_i} =  \HF{S[X]/J}{\vv+\ee_i}$.
    \end{enumerate}
\end{lemma}
\begin{proof}
Since $J$ is $B_X$-saturated, it follows from \cite[Proposition~3.1]{MS04} that for every $i$ there is a nonzerodivisor $\ell_i$ on $S[X]/J$ of degree $\mathbf{e}_i$.
Now, as in the proof of \cite[Proposition~3.5]{Tuy02}, both parts follow from the exact sequence $0 \to S[X]/J \xrightarrow{\cdot \ell_i} S[X]/J \to S[X]/(J+(\ell_i)) \to 0$.
\end{proof}

\begin{lemma}\label{lem:is_injective}
    Let $X=\mathbb{P}^{n_1}\times \mathbb{P}^{n_2}\times \mathbb{P}^{n_3}$ and $r$ be a positive integer satisfying \eqref{eq:ass_for_ts}. Assume that $[I] \in \Hilb_{S[X]}^{h_{r,X}}$ is a $B_X$-saturated ideal. Let 
    \begin{equation}\label{eq:E}
    \mathcal{E} = \{(0,0,0), (1,0,0), (0,1,0), (0,0,1), (1,1,0), (1,0,1), (0,1,1), (1,1,1)\}.
    \end{equation}
    If $I_\mathcal{E} = \bigoplus_{\vv\in \mathcal{E}}I_{\vv}$, then $(I_\mathcal{E})^{\sat} = I$.
\end{lemma}
\begin{proof}
Let $K = (I_\mathcal{E})$. In order to show that $K^{\sat} = I$ we prove that $K^{\sat}\subseteq I$ and that the multigraded Hilbert polynomial of $S[X]/K^{\sat}$ is $r$. This implies that $K^{\sat} = I$ since $h_{r,X}$ is the largest possible multigraded Hilbert function of a $B_X$-saturated ideal with multigraded Hilbert polynomial $r$. 

It follows from $K\subseteq I$ that $K^{\sat} \subseteq I^{\sat} = I$.
On the other hand, $(K^{\sat})_\vv \supseteq K_\vv = I_\vv$ for all $\vv\in \mathcal{E}$. As a result,
\[
(K^{\sat})_\vv = I_\vv
\]
holds for all $\vv \in \mathcal{E}$.
We repeatedly apply Lemma~\ref{lem:properties_of_HF} with $J = K^{\sat}$.
First, from 
\[
\HF{S[X]/K^{\sat}}{1,1,0} = \HF{S[X]/K^{\sat}}{1,1,1} = r
\]
we conclude that $\HF{S[X]/K^{\sat}}{1,1,c} = r$ for all $c\geq 0$. Therefore, for any $c\geq 1$ one gets
\[
r = \HF{S[X]/K^{\sat}}{1,0,1} \leq \HF{S[X]/K^{\sat}}{1,0,c} \leq \HF{S[X]/K^{\sat}}{1,1,c} = r.
\]
As a result, $\HF{S[X]/K^{\sat}}{1,b,c} = r$ for all $b,c\geq 1$.
Finally, for any $b,c\geq 1$, it follows from
\[
r = \HF{S[X]/K^{\sat}}{0,1,1} \leq \HF{S[X]/K^{\sat}}{0,b,c} \leq \HF{S[X]/K^{\sat}}{1,b,c} = r
\]
that $\HF{S[X]/K^{\sat}}{a,b,c} = r$ for all $a,b,c\geq 1$.
This proves that the multigraded Hilbert polynomial of $S[X]/K^{\sat}$ is $r$ and finishes the proof.
\end{proof}

\begin{proposition}\label{prop:ts_test}
Let $X=\mathbb{P}^{n_1}\times \mathbb{P}^{n_2}\times \mathbb{P}^{n_3}$ and $r$ be a positive integer satisfying \eqref{eq:ass_for_ts}. Let $[I]\in \Hilb_{S[X]}^{h_{r,X}}$ and let $K = I + (S[X]_{100})^s + (S[X]_{010})^t + (S[X]_{001})^u$ for positive integers $s,t,u$. If $[I]\in \Slip{r}{X}$ and $s,t,u\ge 2$, then
\[
\dim_\CC \Hom_{S[X]}(K, S[X]/K)_\oo\geq \dim \Slip{r}{X} = r(n_1+n_2+n_3).
\]
\end{proposition}
\begin{proof}
    For non-negative integers $a,b,c$ let $g_{a,b,c}$ be the Hilbert function of 
    \[
    S[X]/\big(I+(S[X]_{100})^a + (S[X]_{010})^b + (S[X]_{001})^c\big).
    \]
    Let $\pi\colon \Hilb_{S[X]}^{h_{r,X}} \to \Hilb_{S[X]}^{g_{s,t,u}}$ and $\theta\colon \Hilb_{S[X]}^{g_{s,t,u}} \to \Hilb_{S[X]}^{g_{2,2,2}}$
    be the natural morphisms given on closed points by $[I'] \mapsto [I'+(S[X]_{100})^s + (S[X]_{010})^t + (S[X]_{001})^u]$ and $[I']\mapsto [I'+(S[X]_{100})^2+(S[X]_{010})^2 + (S[X]_{001})^2]$, respectively. It follows from Lemma~\ref{lem:is_injective} that $\theta\circ \pi$ is injective on the set of $B_X$-saturated ideals. Hence so is $\pi$. Therefore, $\dim \pi(\Slip{r}{X}) = \dim \Slip{r}{X} = r\dim X = r(n_1+n_2+n_3)$. As a result, if $[I]\in \Slip{r}{X}$, then the tangent space to $\Hilb_{S[X]}^{g_{s,t,u}}$ at $[K] = \pi([I])$ has dimension at least $r(n_1+n_2+n_3)$. This tangent space is isomorphic to $\Hom_{S[X]}(K, S[X]/K)_\oo$ by \cite[Proposition~1.6]{HS04}. 
\end{proof}

\begin{remark}
In the language of \cite{Man23}, Lemma~\ref{lem:is_injective} proves that $\mathcal{E}$ as in \eqref{eq:E} is $(r,X)$-sufficient. Therefore, Proposition~\ref{prop:ts_test} is a special case of \cite[Proposition~5.2]{Man23} with $\mathcal{B} = \NN^3$ and $\mathcal{A} = \{\vv\in \NN^3 \mid v_1 \ge s \text{ or } v_2 \ge t \text{ or } v_3 \ge u\}$.
\end{remark}
 Using border apolarity (Theorem~\ref{ba}) and Proposition~\ref{prop:ts_test} we can compute the border rank of the Sch\"onhage triple sum  tensor \eqref{eq:SchonhageTensor} for $(u,v)=(1,2)$.

\begin{theorem}\label{thm:Scho}
    If $T=M_{\langle 1, 1, 4\rangle} \oplus M_{\langle 2,2,1\rangle} \oplus M_{\langle 2,2,1\rangle}  \in \CC^{9}\otimes \CC^{8}\otimes \CC^{8}$, then $\brk(T) =12$.
\end{theorem}
\begin{proof}
One has
\[
\brk(T) \le \brk(M_{\langle 1, 1, 4\rangle})+ \brk(M_{\langle 2,2,1\rangle}) + \brk( M_{\langle 2,2,1\rangle})  = 4 + 4 + 4 = 12.
\]
To obtain the matching lower bound we use border apolarity (Theorem~\ref{ba}). Let $X=\PP^8\times \PP^7\times \PP^7$. We proceed as in the algorithm from \cite[Section~3]{chl19}. That is we
\begin{enumerate}
    \item choose a connected solvable subgroup $\mathbb{B}_T$ of the stabilizer of $T$;
    \item enumerate all possible choices of triples of $\mathbb{B}_T$-fixed codimension $11$ subspaces of $S[X]_{110}$, $S[X]_{101}$, $S[X]_{011}$ contained in $\Ann(T)$;
    \item find all candidate triples, i.e.,  eliminate triples that fail one of the six $(210)$ tests or the $(111$) test (see Subsection~\ref{sub:HFtests}).
\end{enumerate}
All three steps are performed using the implementation of the above algorithm due to Austin Conner which is available at \url{https://github.com/adconner/bapolar}.
We find that there are exactly $12$ candidate triples. Let $\{I_1,\ldots, I_{12}\}$ be ideals generated by these $\mathbb{B}_T$-fixed triples. The degree $111$ part of each $I_i$ is of codimension $11$ in $S[X]_{111}$. Let $K_i = I_i + (S[X]_{100})^2 + (S[X]_{010})^2 + (S[X]_{001})^2$. 
If $\brk(T) \le 11$, then at least one $I_i$ extends to an ideal $[I]$ in $\Slip{11}{X}$. Since $I+(S[X]_{100})^2 + (S[X]_{010})^2 + (S[X]_{001})^2 = K_i$ for some $i=1,\ldots, 12$, it follows from Proposition~\ref{prop:ts_test} that $\max_i \dim_\CC \Hom_{S[X]}(K_i, S[X]/K_i)_\mathbf{0} \ge 11(8+7+7) = 242$. However, a verification in \textit{Macaulay2} \cite{M2} shows that this maximum is $215$.
 \end{proof}

\section{First application of a square test}

In this section we state the versions of the second test of the topological condition relevant for this paper and apply one of them to compute the border rank of $T_{\mathrm{Leit},{5}}\boxtimes W$. See Proposition~\ref{prop:st_general} for the general statement of the test.
We also state the dual versions of these two tests. These are used in the subsequent sections.

\begin{corollary}[Corollary~\ref{c:211test}]\label{c:211testsv}
    If $X=\PP^n\times \PP^n\times\PP^n$ for $n\geq 3$ and $r$ is a positive integer such that $(3n+1)r \leq \dim_\CC S[X]_{111} = (n+1)^3$, then 
    \begin{equation}\label{ineq:211}
    \HF{S[X]/I^2}{2,1,1} \geq  (3n+1)r    
    \end{equation}
	for all $[I]\in \Slip{r}{X}$.
\end{corollary}

\begin{corollary}[Corollary~\ref{c:220test}]\label{c:220testsv}
    If $X=\PP^n\times \PP^n\times \PP^n$ and $r$ is a positive integer such that $(2n+1)r \leq \dim_\CC S[X]_{210} = \binom{n+2}{2}(n+1)$, then 
    \begin{equation}\label{ineq:220}
    \HF{S[X]/I^2}{2,2,0} \geq  (2n+1)r    
    \end{equation}
    for all $[I]\in \Slip{r}{X}.$
\end{corollary}

\begin{remark}
The numbers on the right-hand sides of \eqref{ineq:211} and \eqref{ineq:220} are the expected dimensions of the affine cones of the $r$-th secant varieties of $\PP^n\times \PP^n\times \PP^n$ embedded by $\OO_{(\PP^n)^{\times 3}}(2,1,1)$ and $\PP^n\times \PP^n$ embedded by $\OO_{\PP^n\times \PP^n}(2,2)$, respectively.
\end{remark}

Assume that $X=\PP^n\times \PP^n\times \PP^n$. We say that an ideal $[I]\in \Hilb_{S[X]}^{h_{r,X}}$ \emph{passes the $(211)$ (resp. $(220)$) square test} if \eqref{ineq:211}  (resp. \eqref{ineq:220}) holds. Otherwise, we say that $[I]$ \emph{fails} the corresponding square test. More generally, if $K$ is a $\mathbb{Z}^3$ homogeneous ideal of $S[X]$ such that $\HF{S[X]/K}{\vv} \ge h_{r,X}(\vv)$ for all $\vv\in \mathbb{Z}^3$ and one of the inequalities does not hold, we say that $K$ \emph{fails} the corresponding square test. In this case, there is no ideal $[I']\in \mathrm{Slip}_{r,X}$ with $K \subseteq I'$. We define $(121),(112),(202)$ and $(022)$ square tests in a similar way.

Recall the definitions of $W$ from \eqref{eq:Wtensor} and $T_{\mathrm{Leit},{5}}$ from \eqref{eq:Leitner}.

\begin{theorem}\label{thm:leit}
$\brk(T_{\mathrm{Leit},{5}}\boxtimes W) = 12$.
\end{theorem}
\begin{proof}
    Since $\brk(T_{\mathrm{Leit},{5}})= 6$ (see \cite[Proposition~5.1]{LM17}) and $\brk(W) = 2$, it is enough to show that $\brk(T_{\mathrm{Leit},{5}}\boxtimes W) > 11$.
    We use border apolarity. Let $X = \PP^9\times \PP^9\times \PP^9$. 
    We proceed as in the algorithm recalled in the proof of Theorem~\ref{thm:Scho} with the following modifications:
    \begin{enumerate}
    \item In step (2) we enumerate all quadruples of $\mathbb{B}_T$-fixed codimension $11$ subspaces of $S[X]_{110}$, $S[X]_{101}$, $S[X]_{011}, S[X]_{200}$ contained in $\Ann(T_{\mathrm{Leit},{5}}\boxtimes W)$.
    \item In step (3) we perform all total degree $3$ tests (see Subsection~\ref{sub:HFtests}) on these quadruples.
    \end{enumerate}
There are exactly $16$ such quadruples passing all total degree $3$ tests. All of them fail the $(211)$ square test from Corollary~\ref{c:211testsv}.
\end{proof}

\subsection{Dual formulation of square tests}

In order to perform the $(211)$ and $(220)$ square tests in a convenient way we state them in the dual language of inverse systems. Similar approach was taken in \cite[Proposition~3.5]{chl19} for the $(210)$ and $(111)$ tests. We start with the $(211)$ square test. However, it is simpler to state a ``partial'' $(211)$ square test, namely compute the codimension of the image of the multiplication map $I_{110}\otimes I_{101} \to S[X]_{211}$. In the application below, we first bound the contribution of $I_{200}\cdot I_{011}$ towards the full $(211)$ square test and then perform the ``partial'' one.
Let $f^\wedge\colon A\otimes B\otimes A\otimes C \to \Lambda^2 A\otimes B\otimes C$ be the skew-symmetrization map on the $A$-factors tensored with the identity map on $B\otimes C$. 
Given a subspace $L$ of $A\otimes B\otimes A\otimes C$ the restriction of $f^\wedge$ to $L$ is denoted $f^\wedge_L$.

\begin{lemma}\label{lem:partial211ST}
    Let $I\subseteq S[X]$ be a $\mathbb{Z}^3$-homogeneous ideal. The codimension of the image of the multiplication map $I_{110}\otimes I_{101} \to S[X]_{211}$ is equal to:
    \begin{equation}\label{eq211ST}
    \dim_\CC \ker f^{\wedge}_{I_{110}^\perp\otimes A\otimes C  + A\otimes B \otimes I_{101}^\perp}.
    \end{equation}
    Let $f_{110} = f^{\wedge}_{I_{110}^\perp \otimes A\otimes C}$ and $f_{101} = f^\wedge_{A\otimes B \otimes I_{101}^\perp}$. Then \eqref{eq211ST} is equal to 
    \begin{equation}\label{eq211STver2}
    \dim_\CC \ker f_{110} + \dim_\CC \ker f_{101} + \dim_\CC (\mathrm{im} f_{110} \cap \mathrm{im} f_{101}) - \dim_\CC I_{110}^\perp \dim_\CC I_{101}^\perp.    
    \end{equation}
\end{lemma}
\begin{proof}
  The fact that \eqref{eq211STver2} is equal to \eqref{eq211ST} follows from the fact that $(I_{110}^\perp \otimes A\otimes C)\cap ( A\otimes B \otimes I_{101}^\perp) = I_{110}^\perp \otimes I_{101}^\perp$. Let $d$ be the codimension of the image of $I_{110}\otimes I_{101} \to S[X]_{211}$. It is equal to the dimension of the kernel of the transposed map $
  \widetilde{S[X]}_{211} \to I_{110}^\vee\otimes I_{101}^\vee$.
  The tensor space $A\otimes B\otimes A\otimes C$ is isomorphic to both
  \[
  (I_{110}^\perp\oplus I_{110}^\vee) \otimes (I_{101}^\perp \oplus I_{101}^\vee)= [(I_{110}^\perp\otimes I_{101}^\perp) \oplus  (I_{110}^\perp \otimes I_{101}^\vee)\oplus (I_{110}^\vee\otimes I_{101}^\perp)]\oplus (I_{110}^\vee\otimes I_{101}^\vee)
  \]
  and
  \[
  (S^2 A\otimes B\otimes C)  \oplus (\Lambda^2A \otimes B\otimes C)= \widetilde{S[X]}_{211} \oplus (\Lambda^2 A\otimes B\otimes C),
  \]
and under these isomorphisms $\widetilde{S[X]}_{211} \to I_{110}^\vee\otimes I_{101}^\vee$ is the restriction of the projection of $A\otimes B\otimes A\otimes C$ onto $I_{110}^\vee\otimes I_{101}^\vee$. It follows that $d$ is equal to the dimension of the intersection of  $S^2A\otimes B\otimes C$ with 
  \[
  (I_{110}^\perp\otimes I_{101}^\perp) \oplus ( I_{110}^\perp\otimes I_{101}^\vee) \oplus ( I_{110}^\vee\otimes I_{101}^\perp) = I_{110}^\perp \otimes A\otimes C + A\otimes B \otimes I_{101}^\perp 
  \]
  and hence is equal to \eqref{eq211ST}.
\end{proof}

Now we describe a partial $(220)$ square test in the language of inverse systems. Given a subspace $L\subseteq A\otimes B$, the map $L\otimes A\otimes B \to \Lambda^2 A\otimes \Lambda^2 B$ obtained by restricting the domain of the projection map $A\otimes B\otimes A\otimes B \to \Lambda^2 A\otimes \Lambda^2 B$ is denoted $f^{\wedge \wedge}_L$

\begin{lemma}\label{lem:partial220ST}
    Let $I\subseteq S[X]$ be a $\mathbb{Z}^3$-homogeneous ideal. The codimension of the image of the multiplication map $I_{110}\otimes I_{110} \to S[X]_{220}$ is equal to 
    \begin{equation}\label{eq220ST}
    \dim_\CC \ker f^{\wedge \wedge}_{I_{110}^\perp} -\dim_\CC \Lambda^2 I_{110}^\perp.
    \end{equation}
\end{lemma}
\begin{proof}
  Let $r=\dim_\CC I_{110}^\perp$ and $d$ be the codimension of the image of $I_{110}\otimes I_{110} \to S[X]_{220}$. Then, $d$ is equal to the dimension of the kernel of the transposed map $
  \widetilde{S[X]}_{220} \to I_{110}^\vee\otimes I_{110}^\vee$.
  The tensor space $A\otimes B\otimes A\otimes B$ is isomorphic to both
  \[
  (I_{110}^\perp\oplus I_{110}^\vee) \otimes (I_{110}^\perp \oplus I_{110}^\vee)= [(I_{110}^\perp\otimes I_{110}^\perp) \oplus  (I_{110}^\perp \otimes I_{110}^\vee)\oplus (I_{110}^\vee\otimes I_{110}^\perp)]\oplus (I_{110}^\vee\otimes I_{110}^\vee)
  \]
  and
  \begin{align*}
  &(S^2 A\otimes S^2 B) \oplus [(S^2 A\otimes \Lambda^2 B) \oplus (\Lambda^2 A\otimes S^2 B) \oplus (\Lambda^2 A\otimes \Lambda^2 B)] \\
  &= \widetilde{S[X]}_{220} \oplus [(S^2 A\otimes \Lambda^2 B) \oplus (\Lambda^2 A\otimes S^2 B) \oplus (\Lambda^2 A\otimes \Lambda^2 B)],
\end{align*}
  and under these isomorphisms $\widetilde{S[X]}_{220} \to I_{110}^\vee\otimes I_{110}^\vee$ is the restriction of the projection of $A\otimes B\otimes A\otimes B$ onto $I_{110}^\vee\otimes I_{110}^\vee$. It follows that $d$ is equal to the dimension of the intersection of  $S^2A\otimes S^2 B$ with 
  \[
  (I_{110}^\perp\otimes I_{110}^\perp) \oplus ( I_{110}^\perp\otimes I_{110}^\vee) \oplus ( I_{110}^\vee\otimes I_{110}^\perp) = I_{110}^\perp \otimes A\otimes B + A\otimes B \otimes I_{110}^\perp 
  \]
  and hence is equal to the dimension of the kernel of
  \[
  f\colon I_{110}^\perp \otimes A\otimes B + A\otimes B \otimes I_{110}^\perp \to (S^2 A\otimes \Lambda^2 B) \oplus (\Lambda^2 A\otimes S^2 B) \oplus (\Lambda^2 A\otimes \Lambda^2 B)
  \]
  obtained by embedding the source into $A\otimes B\otimes A\otimes B$ and projecting onto the three irreducible $\mathrm{GL}(A)\times \mathrm{GL}(B)$-submodules.
  Let $f_{\wedge}\colon I_{110}^\perp \otimes A\otimes B \to (S^2 A\otimes \Lambda^2 B) \oplus (\Lambda^2 A\otimes S^2 B)$ and $f_S\colon I_{110}^\perp \otimes A\otimes B \to \Lambda^2 A\otimes \Lambda^2 B$ be the natural maps obtained from $f$ by restricting the domain and projecting onto the direct summands of the codomain. 
  From the equality $\mathrm{im}(f) = \mathrm{im}(f_\wedge) \oplus \mathrm{im}(f_S)$ it follows that
  \begin{align}
  &\dim_\CC \ker f = \dim_\CC [I_{110}^\perp \otimes A\otimes B + A\otimes B \otimes I_{110}^\perp] - \dim_\CC \mathrm{im}(f_\wedge) - \dim_\CC \mathrm{im}(f_S)\notag \\
  &=[2\cdot \dim_\CC I_{110}^\perp \cdot \dim_\CC A\cdot  \dim_\CC B-(\dim_\CC I_{110}^\perp)^2] \notag\\
  &- [\dim_\CC I_{110}^\perp \cdot \dim_\CC A \cdot \dim_\CC B - \dim_\CC \ker (f_\wedge)]  \label{eq0}\\
  &- [\dim_\CC I_{110}^\perp \cdot \dim_\CC A \cdot \dim_\CC B - \dim_\CC \ker (f_S)] \notag\\
  &= \dim_\CC \ker (f_\wedge) + \dim_\CC \ker (f_S) - r^2. \notag 
\end{align}
 Since $I_{110}^\perp \otimes A\otimes B \subseteq A\otimes B\otimes A\otimes B$ and the kernel of $A\otimes B\otimes A\otimes B \to (S^2 A\otimes \Lambda^2 B) \oplus (\Lambda^2 A\otimes S^2 B)$ is equal to $S^2(A\otimes B)$, it follows that $\ker (f_\wedge) = S^2(I_{110}^\perp)$. In particular, its dimension is equal to $\frac{1}{2}r(r+1)$. Therefore, \eqref{eq220ST} follows from \eqref{eq0}.
\end{proof}

\begin{remark}
From the dual restatement of the partial $(220)$ square test in Lemma~\ref{lem:partial220ST} it can be seen that this test for a concise tensor in $(\CC^m)^{\otimes 3}$ and $r=m$ coincides with testing the rank of the Kronecker-Koszul flattening from \cite[Example~3.7]{DM26}.
\end{remark}

\section{Border rank of the direct sum of smallest skew-symmetric Coppersmith-Winograd tensors}

The next tensor for which we compute the border rank is the direct sum of two copies of $T_{\mathrm{skewcw},2}$. The family of tensors $\{T_{\mathrm{skewcw},q}\}_{q\in  2\cdot \mathbb{Z}_>0}$ was introduced in \cite{CGLV22} as skew-symmetric versions of the Coppersmith-Winograd tensors.
It follows from \cite[Theorem~1.6, Lemma~2.8 and Proposition~3.2]{CGLV22} that the border rank of the Kronecker square of $T_{\mathrm{skewcw},2}$ is strictly submultiplicative. Here we show that there is no subadditivity for border rank of the direct sum of two copies of $T_{\mathrm{skewcw},2}$. When enumerating $200$ spaces for an argument based on border apolarity we use the following simple observation.

\begin{lemma}\label{lem:inverseSystem}
    Let $X=\mathbb{P}^{n_1}\times \mathbb{P}^{n_2}\times \mathbb{P}^{n_3}$ and $I\subseteq S[X]$ be a $\mathbb{Z}^3$-homogeneous ideal.
    Let $\mathbf{u}\in  \mathbb{Z}^3$ and $i\in \{1,2,3\}$ be such that $I_{\mathbf{u} + \mathbf{e}_i} = S[X]_{\mathbf{e}_i}\cdot I_{\mathbf{u}}$. Then
    \[
    I_{\mathbf{u}+\mathbf{e}_i}^\perp = \{F \in \widetilde{S[X]}_{\mathbf{u}+\mathbf{e}_i} \mid S[X]_{\mathbf{e}_i}\lrcorner F \subseteq I_{\mathbf{u}}^\perp \}.
    \]
\end{lemma}
\begin{proof}
    If $F\in I_{\mathbf{u}+\mathbf{e}_i}^\perp$, then $(\nu \theta)\lrcorner F = 0$ holds for every $\theta\in I_{\mathbf{u}}$ and for every  $\nu\in S[X]_{\mathbf{e}_i}$.
    Therefore, $\theta \lrcorner (\nu \lrcorner F) = 0$. Since this is true for every $\nu$ and $\theta$, it follows that $S[X]_{\mathbf{e}_i}\lrcorner F \subseteq I_{\mathbf{u}}^\perp$.

    Assume that $F\in \widetilde{S[X]}_{\mathbf{u}+\mathbf{e}_i}$ satisfies $S[X]_{\mathbf{e}_i}\lrcorner F \subseteq I_{\mathbf{u}}^\perp$ and let $\xi\in I_{\mathbf{u}+\mathbf{e}_i}$. By the assumptions it can be written as $\sum_{i=1}^k \nu_i \theta_i$ with $\nu_i\in S[X]_{\mathbf{e}_i}$ and $\theta_i \in I_\mathbf{u}$.
    Therefore,
    \[
    \xi \lrcorner F = \sum_{i=1}^k (\nu_i \theta_i)\lrcorner F = \sum_{i=1}^k \theta_i \lrcorner (\nu_i \lrcorner F) = 0.
    \]
\end{proof}

\begin{theorem}\label{thm:skew}
$\brk(T_{\mathrm{skewcw},2}^{\oplus 2}) = 10$, where $T_{\mathrm{skewcw},{2}}\in \Lambda^3 \CC^3 \subseteq \CC^3\otimes \CC^3\otimes \CC^3$ is the skew-symmetric Coppersmith-Winograd tensor. 
\end{theorem}
\begin{proof}
Let $T =  T_{\mathrm{skewcw},2}^{\oplus 2}$.  
    Since $\brk(T_{\mathrm{skewcw},2})= 5$ by \cite[Proposition~3.2]{CGLV22}, it is enough to show that $\brk(T) > 9$. We use border apolarity (Theorem~\ref{ba}). Let $X = \PP A\times \PP B\times \PP C$.
    Whenever in this proof we write that $I$ is a candidate ideal we mean that $I$ is a $\mathbb{Z}^3$-homogeneous ideal of $S[X]$ such that 
    \begin{enumerate}
        \item $I\subseteq \Ann(T)$
        \item $\HF{S[X]/I}{\mathbf{u}} = 9$ for all $\mathbf{u}$ such that $I$ has a minimal generator of degree $\mathbf{u}$ 
        \item $I$ is $\mathbb{B}_T$-fixed where $\mathbb{B}_T$ is defined below.
    \end{enumerate}
    In particular $I$ might not be a point of $\Hilb_{S[X]}^{h_{9,X}}$. We get elimination by considering only $I_{\mathbf{u}}$ with $\mathbf{u}\in \{110,101,011,200,020\}$. We begin with enumerating possible $I_{110}$ and by symmetry also $I_{101}$ and $I_{011}$.
    Let $E,F\cong \mathbb{C}^3$. Then $A,B,C\cong E\oplus F$. Our tensor lives in $\Lambda^3E \oplus \Lambda^3 F$ and $T(C^\vee) = \Lambda^2 E \oplus \Lambda^2 F$.
    Fix bases $e_1, e_2, e_3$ of $E$ and $f_1, f_2, f_3$ of $F$. Let $\mathbb{B}_E$ (resp. $\mathbb{B}_F$) be the group of all automorphisms of $E$ (resp. $F$) that in the chosen bases are represented by upper-triangular matrices with determinant $1$. Let $\mathbb{B}_T = \mathbb{B}_E\times \mathbb{B}_F$. We look for a $3$-dimensional $\mathbb{B}_T$-fixed subspace in the complement of $T(C^\vee)$ in $(E\oplus F)^{\otimes 2} = T(C^\vee) \oplus S^2E \oplus (E\otimes F) \oplus (F\otimes E) \oplus S^2F$. Let $\mathfrak{b}_T$ be the Lie algebra of $\mathbb{B}_T$. The following directed graphs illustrate the action of $\mathfrak{b}_T$ on $S^2E$ and $E\otimes F$. There are similar graphs for $F\otimes E$ and $S^2F$. We consider the action of the $4$ generators $x_1,x_2,x_3,x_4$ of $\mathfrak{b}_T$ corresponding to the entries on the superdiagonals of the two $3\times 3$ matrices. An edge indicates that there is an element $x_k$ that takes the label of the source to the label of the target. 

    {\tiny 
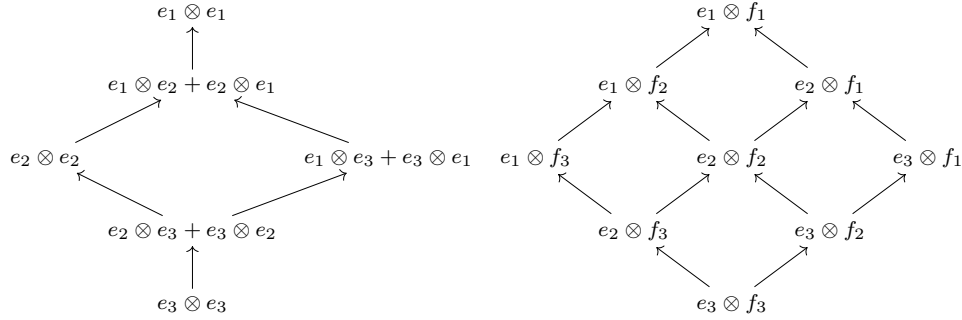
       \begin{center}
        \begin{tikzcd}[cramped, column sep = tiny]
            & e_{1}\otimes e_{1} &  & & & e_{1}\otimes f_{1} & &     \\
            & e_{1}\otimes e_{2} + e_2\otimes e_1 \arrow[u] & &&  e_{1}\otimes f_{2} \arrow[ur] & & e_{2}\otimes f_{1} \arrow[ul] &\\
            e_{2}\otimes e_{2} \arrow[ru] & & e_{1}\otimes e_{3} + e_3\otimes e_1 \arrow[lu]  & e_{1}\otimes f_{3} \arrow[ur] & &   e_{2}\otimes f_{2} \arrow[ru]\arrow[lu] & & e_{3}\otimes f_{1} \arrow[ul] \\
             & e_{2}\otimes e_{3} + e_3\otimes e_2 \arrow[lu] \arrow[ru] & & &   e_{2}\otimes f_{3} \arrow[ur] \arrow[ul] & & e_{3}\otimes f_{2} \arrow[ul] \arrow[ru]& \\
               & e_{3}\otimes e_{3} \arrow[u] &  & &  & e_{3}\otimes f_{3}\arrow[lu] \arrow[ru] & & 
        \end{tikzcd}
    \end{center}
    }
       
    Up to symmetry we may take subspaces of $S^2E \oplus (E\otimes F)\oplus (F\otimes E) \oplus S^2 F$ with the following pure distributions:
    \begin{align*}
        (3000) & \,2 \\
        (2001) & \,1 \\
        (2100) & \,1 \\
        (1101) & \,1 \\
        (1200) & \,2 \\
        (1110) & \,1 \\
        (0300) & \,3 \\
        (0210) & \,2 
    \end{align*}
Here $(d_1d_2d_3d_4) \,k$ means that that there are $k$ $\mathbb{B}_T$-fixed $3$ dimensional subspaces of $S^2E \oplus (E\otimes F)\oplus (F\otimes E) \oplus S^2 F$
for which the dimensions of intersections with the $4$ direct summands are $d_1,d_2,d_3$ and $d_4$, respectively.
There are also choices involving the element $e_1\otimes f_1 + f_1\otimes e_1$ (but not containing $e_1\otimes f_1$). Up to symmetry there are $2$ such choices 
 corresponding to the following pure distributions of the $2$-dimensional complement of $\langle e_1\otimes f_1 + f_1\otimes e_1\rangle$
   \begin{align*}
        (2000) & \,1 \\
        (1001) & \,1
    \end{align*}
    However, as can be checked \textit{Macaulay2}\cite{M2}, only the $3$ choices corresponding to the first two pure distributions pass the $(210)$ and $(120)$ tests. 
    Taking into account the symmetries we get $6$ choices for each of $I^\perp_{011}, I^\perp_{101}$ and $I^\perp_{110}$.
    This gives a total of $216$ choices of triples $(I^\perp_{011}, I^\perp_{101}, I^\perp_{110})$. Among them, there are $22$ that pass the $(111)$ test.
    
    Let $W_{2001}$ be the unique $\mathbb{B}_T$-fixed $9$-dimensional subspace of $(E\oplus F)^{\otimes 2}$ corresponding to the pure distribution $(2001)$.
    Let $W_{3000,L}$ and $W_{3000,R}$ be the two $\mathbb{B}_T$-fixed $9$-dimensional subspaces of $(E\oplus F)^{\otimes 2}$ corresponding to the pure distribution $(3000)$ with $e_2\otimes e_2 \in W_{3000,L}$ and $e_1\otimes e_3 + e_3\otimes e_1 \in W_{3000,R}$. Up to symmetry the $22$ triples can be reduced to the following:
    \begin{align*}
    &(I_{011}^\perp, I_{101}^\perp, I_{110}^\perp) = (W_{3000,L}, W_{3000,L},  W_{3000,L})\\
    &(I_{011}^\perp, I_{101}^\perp, I_{110}^\perp) = (W_{3000,R}, W_{3000,R}, W_{3000,R}) \\
    &(I_{011}^\perp, I_{101}^\perp, I_{110}^\perp) = (W_{3000,L}, W_{3000,R}, W_{3000,R}) \\
    &(I_{011}^\perp, I_{101}^\perp, I_{110}^\perp) = (W_{3000,L}, W_{3000,L}, W_{3000, R}) \\
    &(I_{011}^\perp, I_{101}^\perp, I_{110}^\perp) = (W_{2001}, W_{3000,R}, W_{3000,R} ).
    \end{align*}

    Next we look for possible $200$ spaces. We need a $\mathbb{B}_T$-fixed $9$-dimensional subspace of $S^2(E\oplus F) = S^2 E \oplus (EF) \oplus S^2F$, where $EF$ is spanned by elements of the form $e_i\otimes f_j + f_j\otimes e_i$. The graph representing the $\mathfrak{b}_T$-action on $EF$ is isomorphic to the graph for $E\otimes F$ considered above. Up to symmetry we may take subspaces of $S^2E \oplus EF \oplus  S^2 F$ with the following pure distributions:
     
     \begin{minipage}[t]{.30\textwidth}
     \begin{align*}
        (630) & \, 3 \\
        (621) & \, 2 \\
        (612) & \, 1 \\
        (603) & \, 2 \\
        (540) & \, 3 \\
        (531) & \, 3 \\
        (522) & \, 2 \\
        (513) & \, 2 \\
        (504) & \, 1 \\
     \end{align*}    
     \end{minipage}
     \begin{minipage}[t]{.30\textwidth}
       \begin{align*}
        (450) & \, 3 \\
        (441) & \, 3 \\
        (432) & \, 3\\
        (423) & \, 4\\
        (414) & \, 1\\
        (360) & \, 6\\
        (351) & \, 6\\
        (342) & \, 6 \\
        (333) & \, 12\\
        \end{align*}
        \end{minipage}
        \begin{minipage}[t]{.30\textwidth}
        \begin{align*}
        (270) & \, 2\\
        (261) & \, 3\\
        (252) & \, 3\\
        (180) & \, 1\\
        (171) & \, 2\\
        (090) & \, 1
    \end{align*}
    \end{minipage}
       
Taking into account the symmetries it gives a total of $131$ choices for the $200$ space.
As explained above, up to symmetry, there are two choices of $I_{110}^\perp$: $W_{3000, L}$ and $W_{3000,R}$. Both of them pass the $(210)$ test without excess. 
Therefore, $I_{210}^\perp = \{F\in \widetilde{S[X]}_{210}\mid A^\vee\lrcorner F \in I_{110}^\perp\}$  by Lemma~\ref{lem:inverseSystem}.
Both $W_{3000, L}$ and $W_{3000,R}$ contain $e_1\otimes e_1$ and $e_2\otimes e_1$ and $W_{3000,R}$ contains also $e_3\otimes e_1$. It follows that 
$I_{210}^\perp$ contains $(e_1e_i)\otimes e_1$ and $(e_2e_i)\otimes e_1$ for $i=1,2,3$. As a result, $I_{200}^\perp$ contains $e_1e_i$ and $e_2e_i$ for $i=1,2,3$.
So a necessary condition for a $200$ space to pass the $(210)$ test (including both the newly added $200$ space and the already fixed $110$ space) is that its pure distribution is of the form $(\ge 5,?,?)$. Furthermore, if $I_{110}^\perp = W_{3000,R}$, then $(e_3e_i)\otimes e_1$ is in $I_{210}^\perp$ so $e_3e_i \in I_{200}^\perp$ and the pure distribution is of the form $(6,?,?)$. This gives at most $8\cdot 4 = 32$ quadruples $(I_{011}^\perp, I_{101}^\perp, I_{110}^\perp, I_{200}^\perp)$ passing all total degree $3$ tests with $I_{110}^\perp = W_{3000,R}$ and at most $1\cdot 19 =19$ such quadruples with $I_{110}^\perp = W_{3000,L}$ for a total of $51$ quadruples. Here we do an explicit computer-free elimination based on the $(211)$ square test for the quadruple $(W_{3000,R}, W_{3000,R}, W_{3000,R}, W_{603,R}) $ where $W_{603,R}$ is the unique $\mathbb{B}_T$-fixed $200$ space with pure distribution $(603)$ and containing $f_1f_3$. We discuss the remaining $50$ quadruples at the end of the proof. In order to lower bound the contribution of $I_{200}$ and $I_{011}$ towards the $(211)$ square test we switch to the ideal side. 

Let $A = A_1\oplus A_2$ with $A_1 = \langle a_1,a_2,a_3\rangle$ and $A_2 = \langle a'_1,a'_2,a'_3\rangle$ where we identify $A_1$ and $A_2$ with $E$ and $F$, respectively. More precisely, for $i=1,2,3$, we identify $a_i$ with $e_i$ and $a'_i$ with $f_i$. Let $B_1,B_2, C_1,C_2$ be defined in a similar way. Let 
\[
S[X] = \CC[\alpha_1,\alpha_2,\alpha_3, \alpha'_1, \alpha'_2, \alpha'_3, \beta_1, \beta_2, \beta_3, \beta_1', \beta_2', \beta_3', \gamma_1, \gamma_2, \gamma_3, \gamma_1', \gamma_2', \gamma_3']
\]
where $\alpha_1,\alpha_2,\alpha_3$ is the dual basis for $A_1^\vee$, $\alpha_1', \alpha_2',\alpha_3'$ is the dual basis for $A_2^\vee$ and similarly for the other variables.

\begin{claim}\label{cla:reducingST}
    If $(I_{011}^\perp, I_{101}^\perp, I_{110}^\perp, I_{200}^\perp) = (W_{3000,R}, W_{3000,R}, W_{3000,R}, W_{603,R})$, then 
    \[
    \dim_\CC S[X]_{211}/(I_{110}I_{101}) \ge \dim_\CC S[X]_{211}/(I_{110}I_{101}+ I_{200}I_{011}) + 18. 
    \]
\end{claim}
\begin{proof}
Since $I_{110}^\perp = W_{3000,R}$, it follows that $I_{110}$ is spanned by $A_1^\vee B_2^\vee$, $A_2^\vee B_1^\vee$,  $\langle \alpha_2\beta_2, \alpha_2\beta_3+\alpha_3\beta_2, \alpha_3\beta_3\rangle$ and $\langle \alpha_1'\beta_1', \alpha_1'\beta_2' + \alpha_2'\beta_1', \alpha_2'\beta_2', \alpha_1'\beta_3'+\alpha_3'\beta_1', \alpha_2'\beta_3'+\alpha_3'\beta_2', \alpha_3'\beta_3'\rangle$.
$I_{101}$ and $I_{011}$ have analogous bases. $I_{200}$ is spanned by $A_1^\vee A_2^\vee$ and $((\alpha'_2)^2, \alpha'_2\alpha'_3, (\alpha_3')^2)$. Therefore, $I_{200}\cdot I_{011}$ contains
\begin{align}
&\alpha_1\cdot \langle \alpha'_1, \alpha'_2, \alpha'_3\rangle \cdot \langle \beta_2\gamma_2, \beta_2\gamma_3+\beta_3\gamma_2, \beta_3\gamma_3\rangle \textrm{ and } \notag\\
&\langle (\alpha'_2)^2, \alpha'_2\alpha'_3, (\alpha'_3)^2\rangle \cdot \langle \beta_1'\gamma_1', \beta_1'\gamma_2'+\beta_2'\gamma'_1, \beta_1'\gamma_3'+\beta_3'\gamma_1'\rangle. \label{eq:211contribution}    
\end{align}
We claim that the $18$ dimensional space spanned by these elements has trivial intersection with $I_{110} I_{101}$.
Observe that $I_{110}I_{101}\subseteq S[X]_{211}$ has no elements with a nonzero coefficient in front of any monomial in $\alpha_1 \cdot \langle \alpha'_1, \alpha'_2,\alpha'_3\rangle \cdot \langle \beta_2\gamma_2, \beta_2\gamma_3, \beta_3\gamma_3\rangle$.
In order to prove the claim it is thus sufficient to show that the $9$-dimensional vector space \eqref{eq:211contribution} has trivial intersection with 
the subspace of $S[X]_{211}$ spanned by all pairwise products of elements from the following two vector subspaces of $S[X]$:
\begin{align*}
&\langle \alpha_1'\beta_1', \alpha_1'\beta_2' + \alpha_2'\beta_1', \alpha_2'\beta_2', \alpha_1'\beta_3'+\alpha_3'\beta_1', \alpha_2'\beta_3'+\alpha_3'\beta_2', \alpha_3'\beta_3' \rangle \subseteq S[X]_{110}   \\
&\langle \alpha_1'\gamma_1', \alpha_1'\gamma_2' + \alpha_2'\gamma_1', \alpha_2'\gamma_2', \alpha_1'\gamma_3'+\alpha_3'\gamma_1', \alpha_2'\gamma_3'+\alpha_3'\gamma_2', \alpha_3'\gamma_3'\rangle \subseteq S[X]_{101}.
\end{align*}
Let $F$ be an element of this intersection. Let $S'=\CC[\alpha'_1, \alpha'_2, \alpha'_3, \beta_1', \beta_2', \beta_3']$ be a $\mathbb{Z}^2$-graded polynomial ring with $\deg(\alpha_i') = (1,0)$ and $\deg(\beta_i') = (0,1)$ and consider the ring homomorphism $\pi\colon S[X]\to S'$ given by sending all non-primed variables of $S[X]$ to zero, $\alpha'_i$'s and $\beta_i'$'s to themselves and $\gamma'_i\mapsto \beta_i'$. Then $\pi(F)$ belongs to the degree $(2,2)$ part of the intersection of ideals
\[
[((\alpha_2')^2, \alpha_2'\alpha_3', (\alpha_3')^2)(\beta_1')(\beta_1', \beta_2',\beta_3')] \cap (\alpha'_1\beta'_1, \alpha'_1\beta'_2+\alpha'_2\beta'_1, \alpha'_1\beta'_3 + \alpha'_3\beta'_1, \alpha'_2\beta'_2, \alpha'_2\beta'_3+\alpha'_3\beta'_2, \alpha'_3\beta'_3)^2.
\]
This implies that $\pi(F) =0$ since every element of degree $(2,2)$ in $(\alpha'_1\beta'_1, \alpha'_1\beta'_2+\alpha'_2\beta'_1, \alpha'_1\beta'_3 + \alpha'_3\beta'_1, \alpha'_2\beta'_2, \alpha'_2\beta'_3+\alpha'_3\beta'_2, \alpha'_3\beta'_3)^2$ is stable under the involution of $S'$ that swaps $\alpha_i'$ with $\beta_i'$ but there is no nonzero element of degree $(2,2)$ in $((\alpha_2')^2, \alpha_2'\alpha_3', (\alpha_3')^2)\cdot (\beta_1')\cdot (\beta_1', \beta_2',\beta_3')$ with this property.
Since $\pi$ is injective when restricted to \eqref{eq:211contribution}, it follows that $F=0$.
\end{proof}

We proceed to upper bound $\dim_\CC S[X]_{211}/(I_{110}I_{101})$ using Lemma~\ref{lem:partial211ST}.
By symmetry $\dim_\CC \ker f_{110} = \dim_\CC \ker f_{101}$. Furthermore, since the $B$ factor does not contribute to the kernel of $f_{101}$, the latter dimension is equal to $\dim_\CC B$ times the codimension of the image of the $(201)$ map. As stated above, $I_{101}$ passes the $(201)$ test without excess. Therefore
\begin{equation}\label{eq:dimKer}
\dim_\CC \ker f_{110} + \dim_\CC \ker f_{101} = 2\cdot 6\cdot 9 = 108.    
\end{equation}

We are left with establishing an upper bound on the dimension of the intersection of the image of $f_{110}$ with the image of $f_{101}$.
Due to the decompositions $I_{110}^\perp = (I_{110}^\perp\cap A_1\otimes B_1) \oplus (I_{110}^\perp \cap A_2\otimes B_2)$ 
and $B\otimes C = (B_1\otimes C_1)\oplus (B_1\otimes C_2)\oplus (B_2\otimes C_1)\oplus (B_2\otimes C_2)$, the image of $f_{110}$ is equal to the sum of the images of 
\begin{align*}
    &f_{110}^{1:11}\colon (I_{110}^\perp \cap A_1\otimes B_1) \otimes A_1\otimes C_1 \to \Lambda^2 A_1 \otimes B_1\otimes C_1\\
    &f_{110}^{1:12}\colon (I_{110}^\perp \cap A_1\otimes B_1) \otimes A_1\otimes C_2 \to \Lambda^2 A_1 \otimes B_1\otimes C_2\\
    &f_{110}^{1:21}\colon (I_{110}^\perp \cap A_1\otimes B_1) \otimes A_2\otimes C_1 \to (A_1\otimes A_2 + A_2\otimes A_1) \otimes B_1\otimes C_1 \\
    &f_{110}^{1:22}\colon (I_{110}^\perp \cap A_1\otimes B_1) \otimes A_2\otimes C_2 \to (A_1\otimes A_2 + A_2\otimes A_1) \otimes B_1\otimes C_2\\
    &f_{110}^{2:11}\colon (I_{110}^\perp \cap A_2\otimes B_2) \otimes A_1\otimes C_1 \to (A_1\otimes A_2 + A_2\otimes A_1) \otimes B_2\otimes C_1\\
    &f_{110}^{2:12}\colon (I_{110}^\perp \cap A_2\otimes B_2) \otimes A_1\otimes C_2 \to (A_1\otimes A_2 + A_2\otimes A_1) \otimes B_2\otimes C_2\\
    &f_{110}^{2:21}\colon (I_{110}^\perp \cap A_2\otimes B_2) \otimes A_2\otimes C_1 \to \Lambda^2 A_2 \otimes B_2\otimes C_1 \\
    &f_{110}^{2:22}\colon (I_{110}^\perp \cap A_2\otimes B_2) \otimes A_2\otimes C_2 \to \Lambda^2 A_2 \otimes B_2\otimes C_2
\end{align*}
which are obtained from $f_{110}$ by restricting the domain and codomain. There is a similar decomposition of the image of $f_{101}$. Since the intersection of the image of $f_{110}$ with $\Lambda^2 A_i \otimes B\otimes C$ is contained in $\Lambda^2 A_i \otimes B_i\otimes C$, it follows that $f_{110}^{1:12}$ and $f_{110}^{2:21}$ do not contribute towards the intersection of the images of $f_{110}$ and $f_{101}$. The dimensions of the codomains of $f_{110}^{1:11}$ and $f_{110}^{2:22}$ are $27$, so
\begin{equation}\label{eq:wedge2} 
\dim_\CC (\mathrm{im} f_{110}^{1:11} \cap \mathrm{im} f_{101}^{1:11}) \le 27 \text{ and } \dim_\CC (\mathrm{im} f_{110}^{2:22} \cap \mathrm{im} f_{101}^{2:22})\le 27.
\end{equation}

The domains of $f_{110}^{2:11}$ and $f_{101}^{1:22}$ are contained in $(A_2\otimes B) \otimes (A_1 \otimes C)$ and $(A_1\otimes C) \otimes (A_2\otimes B)$, respectively. Since the antisymmetrization map $A\otimes A\to \Lambda^2 A$ when restricted to $A_1\otimes A_2$ or $A_2\otimes A_1$ is injective, it follows that 
\begin{align}
    \dim_\CC (\mathrm{im} f_{110}^{1:22} \cap \mathrm{im} f_{101}^{2:11}) &= \dim_\CC (\mathrm{im} f_{110}^{2:11} \cap \mathrm{im} f_{101}^{1:22}) \notag \\
    &= \dim_\CC (I_{110}^\perp\cap A_2\otimes B_2)\dim_\CC (I_{101}^\perp \cap A_1\otimes C_1) = 3\cdot 6 = 18 \label{eq:2:11}
\end{align}
where the first equality follows by symmetry.

Similarly,
\begin{align}\label{eq:2:12}
   \dim_\CC(\mathrm{im} f_{110}^{2:12} \cap \mathrm{im} f_{101}^{2:12}) &= \dim_\CC (A_1)  \dim_\CC ([(I_{110}^\perp\cap A_2\otimes B_2)\otimes C_2] \cap [(I_{101}^\perp\cap A_2\otimes C_2)\otimes B_2]) 
   \notag \\
   &= 3\cdot \dim_\CC \langle T_2\rangle  = 3
\end{align}
where the second equality follows from the fact that every element of $[(I_{110}^\perp\cap A_2\otimes B_2)\otimes C_2] \cap [(I_{101}^\perp\cap A_2\otimes C_2)\otimes B_2]$ is antisymmetric with respect to the $A_2$ and $B_2$ factors and antisymmetric with respect to the $A_2$ and $C_2$ factors. Thus, it is a multiple of $T_{\mathrm{skewcw},2}\in A_2\otimes B_2\otimes C_2$.

\begin{claim}\label{cl:1:21}
    $\dim_\CC (\mathrm{im} f_{110}^{1:21} \cap \mathrm{im} f_{101}^{1:21}) \le 39$
\end{claim}
\begin{proof}
    The images of the following $39$ elements are in the intersection. Here $j=1,2,3.$
    \begin{align*}
        &(a_1\otimes b_i)\otimes (a'_j\otimes c_k)  = (a_1\otimes c_k) \otimes (a'_j\otimes b_i)&\, i,k = 1,2,3\\
        &(a_i\otimes b_1)\otimes (a'_j\otimes c_1) = (a_i\otimes c_1) \otimes (a'_j\otimes b_1) &\, i=2,3\\
        &(a_2\otimes b_3 - a_3\otimes b_2) \otimes (a'_j \otimes c_1) = (a_2\otimes c_1)\otimes (a'_j \otimes b_3) - (a_3\otimes c_1)\otimes (a_j'\otimes b_2) \\
        &(a_2\otimes b_1)\otimes (a'_j\otimes c_3) - (a_3\otimes b_1)\otimes (a'_j\otimes c_2) = (a_2\otimes c_3 - a_3\otimes c_2) \otimes (a'_j\otimes b_1).     \end{align*}
    It is also straightforward to verify that they are linearly independent.
Suppose that the intersection is larger. Then there is a nonzero linear combination $F$ of the following elements (with $j=1,2,3$):
\begin{align*}
    & (a_3\wedge a_j') \otimes b_1 \otimes c_k &\, k=2,3\\
    & (a_2\wedge a_j') \otimes b_1 \otimes c_2 \\
    & (a_2\wedge a_j')\otimes b_3\otimes c_k - (a_3\wedge a_j')\otimes  b_2\otimes c_k &\, k=2,3
\end{align*}
that belongs to the image of $f_{101}^{1:21}$. Since none of these $15$ elements depends on $a_1$ or $c_1$, it follows that $F$ belongs to the image of $\langle a_2\otimes c_3 - a_3\otimes c_2\rangle\otimes A_2\otimes B_1$ under $f_{101}^{1:21}$.
In particular, the coefficient of $F$ in front of $(a_3\wedge a_j')\otimes b_1 \otimes c_2$ is opposite to its coefficient in front of $(a_2\wedge a_j')\otimes b_1 \otimes c_3$ and thus is equal to zero. Since no element in  $\langle a_2\otimes c_3 - a_3\otimes c_2\rangle\otimes A_2\otimes B_1$ has a nonzero coefficient in front of any tensor of the form $a_i\otimes a_j' \otimes b_k\otimes c_i$ with $i=2,3$, $j=1,2,3$ and $k=1,2,3$, it follows that $F = 0$.
\end{proof}

After these preparations we can conclude the proof that no ideal $[I]$ in $\Slip{9}{X}$ satisfies  $(I_{011}^\perp, I_{101}^\perp, I_{110}^\perp, I_{200}^\perp) = (W_{3000,R}, W_{3000,R}, W_{3000,R}, W_{603,R})$. Indeed, for any such ideal by Corollary~\ref{c:211testsv} we have $\HF{S[X]/I^2}{2,1,1} \ge 9(\dim_\CC A+\dim_\CC B+\dim_\CC C-2) = 144$. Therefore, by Claim~\ref{cla:reducingST}, the codimension of the image of $I_{110}\otimes I_{101} \to S[X]_{211}$ is at least $144 + 18 = 162$. Hence 
\[
\dim_\CC [\mathrm{im}(f_{110}) \cap \mathrm{im}(f_{101})] \ge 162 - 108 + 81 = 135
\]
follows from Lemma~\ref{lem:partial211ST} and equation \eqref{eq:dimKer}.
However, by Claim~\ref{cl:1:21} and equations \eqref{eq:wedge2}, \eqref{eq:2:11} and \eqref{eq:2:12}, this dimension is equal to at most 
\[
2\cdot 27 + 2\cdot 18 + 39 + 3 = 132.
\]

This finishes the proof that $(I_{011}, I_{101}, I_{110}, I_{200})$ fails the $(211)$ square test and so does not extend to an ideal in $\Slip{9}{X}$. $49$ out of the other $50$ quadruples can be eliminated in a similar way using \textit{Macaulay2} \cite{M2}. The remaining quadruple is $(I_{011}^\perp, I_{101}^\perp, I_{110}^\perp, I_{200}^\perp) = (W_{3000,R}, W_{3000,R}, W_{3000,R}, W_{630,LL})$ where $W_{630,LL}$ is the unique $\mathbb{B}_T$-fixed $200$ space with pure distribution $(630)$ and containing $e_1\otimes f_3 + f_3\otimes e_1$. This one passes the $(211)$ square test. By symmetry, we get unique $\mathbb{B}_T$-fixed $020$ space $W_{630,LL}$ such that $(I_{011}^\perp, I_{101}^\perp, I_{110}^\perp, I_{020}^\perp) = (W_{3000,R}, W_{3000,R}, W_{3000,R}, W_{630,LL})$ passes the $(121)$ square test. The quintuple
\[
(I_{011}^\perp, I_{101}^\perp, I_{110}^\perp, I_{200}^\perp, I_{020}^\perp) = (W_{3000,R}, W_{3000,R}, W_{3000,R}, W_{630,LL}, W_{630,LL})
\]
fails the $(220)$ square test from Corollary~\ref{c:220testsv}.
\end{proof}

\begin{remark}
Since $W$ is a degeneration of $M_{\langle 1\rangle}^{\oplus 2}$, Theorem~\ref{thm:leit} proves also that $\brk(T_{\mathrm{Leit},{5}}^{\oplus 2}) = 12$.
On the other hand, Theorem~\ref{thm:skew} cannot be strengthened in a similar way. The border rank of $T_{\mathrm{skewcw},2}\boxtimes W$ is $9$ by \cite[Proposition~8.4]{HL26}.
\end{remark}

\begin{remark}
    In fact the group $\mathbb{B}_T$ used in the above proof is not a Borel-subgroup of the stabilizer of $T_{\mathrm{skewcw},2}\oplus T_{\mathrm{skewcw},2}$. It can be enlarged as the following one-dimensional torus also stabilizes $T_{\mathrm{skewcw},2}$ but does not belong to $\mathbb{B}_E$:
    \begin{align*}
    t\cdot a_i = t^{\delta_{1i}}a_i\\
    t\cdot b_i = t^{-1}t^{\delta_{1i}}b_i\\
    t\cdot c_i = t^{\delta_{1i}}c_i.
    \end{align*}
    Here $\delta_{1i}$ is the Kronecker delta. Using this bigger group one can eliminate the $I_{110}$ spaces with $I_{110}^\perp$ containing $\langle e_1\otimes f_1 + f_1\otimes e_1\rangle$ but not containing $e_1\otimes f_1$ even before performing the $(210)$ tests.
\end{remark}

\begin{remark}\label{rmk:about_techniques}
The border ranks of the tensors from Theorems~\ref{thm:Scho}, \ref{thm:leit} and  \ref{thm:skew} cannot be computed using Koszul flattenings \cite{LO15} or weak border apolarity (Theorem~\ref{wba}). More precisely,
\begin{itemize}
 \item The best lower bounds for the border rank coming from $p$-th Koszul flattenings (both applied directly to the given tensors and to their restrictions to random subspaces of one of the three factors) are  $10$, $11$ and $9$, respectively.
\item There are $\mathbb{Z}^3$-homogeneous ideals contained in the annihilators of the given tensors which are points in the multigraded Hilbert schemes corresponding to $r = 11, 11, 9$, respectively.
\end{itemize}
\end{remark}

\section{Example: \texorpdfstring{$3\times 3$}{3x3} matrix multiplication tensor}\label{s:m3}
Using the square test we can also recover the current best known lower bound for the border rank of the $3\times 3$ matrix multiplication tensor \cite[Theorem~1.1]{chl19}. Let $X = \PP^8\times \PP^8\times \PP^8$.

Let $U,V,W\cong \mathbb{C}^3$ and let $A = U^\vee\otimes V$, $B=V^\vee\otimes W$ and $C=W^\vee\otimes U$. Fix bases $\{u_1,u_2,u_3\}$, $\{v_1, v_2, v_3\}$, $\{w_1,w_2,w_3\}$ for $U,V$ and $W$, respectively with dual bases $\{u^1, u^2, u^3\}$, $\{v^1, v^2, v^3\}$, $\{w^1, w^2, w^3\}$. Using this notation, $M_{\langle 3\rangle} = \sum_{i,j,k = 1}^3 (u^i\otimes v_j)\otimes (v^j\otimes w_k)\otimes (w^k\otimes u_i)$. Let  $\mathbb{B}_U$ (resp. $\mathbb{B}_V$, resp. $\mathbb{B}_W$) be the group of all automorphisms of $U$ (resp. $V$, resp. $W$) that in the chosen bases are represented by upper-triangular matrices. Let $\mathbb{B} = \mathbb{B}_U\times \mathbb{B}_V\times \mathbb{B}_W$. As observed in the original proof in \cite{chl19}, there are $8$ choices of $\mathbb{B}$-fixed $16$-dimensional spaces $I_{110}^\perp$ containing $M_{\langle 3\rangle}(C^\vee)$ and passing the $(210)$ and $(120)$ tests. All of them fail the $(220)$ square test. Here we present this explicitly for one choice of $I_{110}^\perp$ using Corollary~\ref{c:220testsv} and Lemma~\ref{lem:partial220ST}.

Let $a^i_j = u^i\otimes v_j \in A$ and define $b^{i}_j$ and $c^{i}_j$ in a similar way. 
Let $\{\alpha^{i}_j\}_{i,j}$, $\{\beta^i_j\}_{i,j}$ and $\{\gamma^{i}_j\}_{i,j}$ be the dual bases. 
We consider the $110$ space given by $I_{110}^\perp = M_{\langle 3 \rangle}(C^\vee) \oplus E'_{110}$ where $E'_{110} \oplus \langle a^3_1b^1_1 + a^3_2b^2_1 + a^3_3b^3_1\rangle = \langle a^3_ib^j_1 \mid (i,j) \ne (3,1) \rangle$. By Corollary~\ref{c:220testsv} and Lemma~\ref{lem:partial220ST} one may exclude this $I_{110}^\perp$ by showing that the dimension of the kernel of $f^{\wedge \wedge}_{I_{110}^\perp}$ is smaller than $16\cdot (2\cdot 9-1) + \frac{1}{2}\cdot 16\cdot 15 = 392$.
Below we decompose $I_{110}^\perp$ into two $8$-dimensional subspaces $D\oplus E$. Then we need to show that

\begin{equation}\label{eq:goal}
    \dim_\CC \ker f^{\wedge\wedge}_{E} + \dim_\CC \ker [D\otimes (A\otimes B) \to (\Lambda^2 A\otimes \Lambda^2B) / (\mathrm{im} f^{\wedge \wedge}_E)] < 392.
\end{equation}

We first compute the kernel of $f^{\wedge \wedge}_{M_{\langle 3\rangle}(C^\vee)}$.
Using the isomorphism $\Lambda^2(U^\vee) \cong U$ we can decompose the source into irreducible $\mathrm{SL}(U)\times \mathrm{SL}(V)\times \mathrm{SL}(W)$-modules as follows:

\begin{align*}
&M_{\langle 3\rangle}(C^\vee) \otimes A\otimes B \cong (U^\vee\otimes 1_V \otimes W)\otimes (U^\vee\otimes V \otimes V^\vee\otimes W) \\
&\cong (S^2(U^\vee)\oplus \Lambda^2(U^\vee))\otimes (1_V \oplus \mathrm{sl}(V))\otimes (S^2 W\oplus \Lambda^2 W)\\
&\cong (S^2(U^\vee) \otimes S^2W) \oplus (S^2(U^\vee)\otimes \mathrm{sl}(V) \otimes S^2W) \\
&\oplus (S^2(U^\vee) \otimes \Lambda^2W) \oplus (S^2(U^\vee)\otimes \mathrm{sl}(V) \otimes \Lambda^2W)\\
&\oplus (U \otimes S^2W) \oplus (U\otimes \mathrm{sl}(V) \otimes S^2W) \oplus (U \otimes \Lambda^2W) \oplus (U\otimes \mathrm{sl}(V) \otimes \Lambda^2W).
\end{align*}
By Schur's lemma it is enough to compute the images of the $8$ highest weight vectors, one for each irreducible submodule, to conclude that the kernel of $f^{\wedge \wedge}_{M_{\langle 3\rangle}(C^\vee)}$ is isomorphic to $(S^2U^\vee\otimes \Lambda^2W) \oplus (U\otimes S^2W)$. Observe that as a subspace of $M_{\langle 3\rangle}(C^\vee) \otimes A\otimes B$ this kernel is equal to 
$\Lambda^2(M_{\langle 3\rangle}(C^\vee))$. 
Let 
\[
D = M_{\langle 3\rangle}(\langle \gamma^i_j \mid i,j=1,2,3, (i,j)\ne  (1,3) \rangle) \subseteq M_{\langle 3\rangle}(C^\vee)
\]
and 
\[
E= \langle a^3_ib^j_1 \mid (i,j) \ne (3,1) \rangle.
\]
It is convenient to replace the decomposition $M_{\langle 3\rangle}(C^\vee)\oplus E'_{110}$ by $D \oplus E$.
It follows from the above, that
\begin{equation}\label{eq:kerD}
\ker f^{\wedge \wedge}_{D} =  \Lambda^2 D.
\end{equation}

Now we analyze the kernel of $f^{\wedge \wedge }_{E}$. Let $A_1 = \langle a^3_i \mid i=1,2,3\rangle$, $A_0 = \langle a^i_j \mid i=1,2, j=1,2,3\rangle$, $B_1 = \langle b^i_1 \mid i=1,2,3\rangle$ and $B_0 = \langle b^i_j \mid i=1,2,3, j=2,3\rangle$. In this way, $A\otimes B$ and $\Lambda^2 A\otimes \Lambda^2 B$ become $\mathbb{Z}^2$-graded vector spaces.
The image of $f^{\wedge \wedge }_{E}$ is equal to $(\Lambda^2A \otimes \Lambda^2 B)_{22} \oplus (\Lambda^2 A \otimes \Lambda^2 B)_{21} \oplus (\Lambda^2 A \otimes \Lambda^2 B)_{12} \oplus L_{11}$
where $L_{11}\subseteq (\Lambda^2 A \otimes \Lambda^2 B)_{11}$ is the image of the (injective) restriction $E\otimes A_0\otimes B_0 \to (\Lambda^2 A \otimes \Lambda^2 B)_{11}$ of $f^{\wedge \wedge}_E$. In particular, $\mathrm{codim}_{(\Lambda^2 A \otimes \Lambda^2 B)_{11}} L_{11} = 36$. Therefore, the image of $f^{\wedge \wedge }_{E}$ has dimension $3\cdot 3+3\cdot 18 + 18\cdot 3 + (18\cdot 18 - 36) = 405$.
Since the source has dimension $8\cdot 81$, it follows that 
\begin{equation}\label{eq:kerE}
    \dim_\CC \ker f^{\wedge\wedge}_{E} = 243.
\end{equation}
Therefore, to prove \eqref{eq:goal} it is sufficient by \eqref{eq:kerE} to show that
\[
\dim_\CC \ker [D\otimes (A\otimes B)\to (\Lambda^2 A\otimes \Lambda^2B)/(\mathrm{im}f^{\wedge \wedge}_E)] < 149.
\]
Since the $64$-dimensional subspace $D\otimes E$ of the source is in the kernel, this can be restated as 
\begin{align*}
\dim_\CC &\ker [D\otimes (A\otimes B)/E \to (\Lambda^2 A\otimes \Lambda^2B)/(\mathrm{im}f^{\wedge \wedge}_E)]< 85.
\end{align*}
Here and below, when we consider $(A\otimes B)/Q$ where $Q\subseteq A\otimes B = \widetilde{S[X]}_{110}$ is spanned by monomials, we always identify $(A\otimes B)/Q$ with the subspace of $A\otimes B = \widetilde{S[X]}_{110}$ spanned by the monomials not in $Q$.

Due to \eqref{eq:kerD}, the image of $\Lambda^2 D$ in $D\otimes (A\otimes B)/E$ is also in the kernel. 
Recall that $D$ is spanned by the following elements:

\begin{minipage}[t]{.30\textwidth}
     \begin{align*}
T^1_1 = a^1_1b_1^1 + a^1_2b^2_1 + a^1_3b^3_1\\
T^2_1 = a^1_1b^1_2 + a^1_2b^2_2 + a^1_3b^3_2\\
T^3_1 = a^1_1b^1_3 + a^1_2b^2_3 + a^1_3b^3_3\\
\end{align*}
\end{minipage}
\begin{minipage}[t]{.30\textwidth}
\begin{align*}
T^1_2 = a^2_1b^1_1 + a^2_2b^2_1 + a^2_3b^3_1\\
T^2_2 = a^2_1b^1_2 + a^2_2b^2_2 + a^2_3b^3_2\\
T^3_2 = a^2_1b^1_3 + a^2_2b^2_3 + a^2_3b^3_3\\
\end{align*}
\end{minipage}
\begin{minipage}[t]{.30\textwidth}
\begin{align*}
{} \\
T^2_3 = a^3_1b^1_2 + a^3_2b^2_2 + a^3_3b^3_2\\
T^3_3 = a^3_1b^1_3 + a^3_2b^2_3 + a^3_3b^3_3
\end{align*}
\end{minipage}

\noindent
Let $E' = E + \langle a^1_1b_1^1, a^2_1b^1_1, a^1_1b^1_2, a^2_1b^1_2, a^3_1b^1_2, a^1_1b^1_3, a^2_1b^1_3, a^3_1b^1_3\rangle$ be obtained by adding to $E$ the subspace spanned by the first monomials appearing in the above generators. We can identify $D\otimes (A\otimes B)/E$ with $(D\otimes D) \oplus [D\otimes (A\otimes B)/E']$ and further with $S^2 D \oplus \Lambda^2 D \oplus [D\otimes (A\otimes B)/E']$.
Now taking into account that $\Lambda^2 D$ is mapped to zero we need to show that
\begin{equation}\label{eq:goal2}
\dim_\CC \ker [S^2D \oplus [D\otimes (A\otimes B)/E'] \to (\Lambda^2 A \otimes \Lambda^2 B)/(\mathrm{im}f^{\wedge \wedge}_E)] < 57.
\end{equation}

Since $\dim_\CC (E')_{10} = 2$, $\dim_\CC (E')_{01} = 2$ and $\dim_\CC (E')_{11} = 8$, it follows that $((A\otimes B)/E')_{10} \cong \mathbb{C}^{16}$, $((A\otimes B)/E')_{01} \cong \mathbb{C}^{16}$ and $((A\otimes B)/E')_{11} \cong \mathbb{C}^1$. Observe that $T_1^1, T^1_2 \in A_{0}\otimes B_1$, $T^2_3, T^3_3 \in A_1\otimes B_0$ and $T^2_1, T^2_2, T^3_1, T^3_2 \in A_0\otimes B_0$. In particular, $D$ is also $\mathbb{Z}^2$-graded.
Since the image of $\mathrm{im}f^{\wedge \wedge}_E$ is contained in degrees $(2,2), (2,1), (1,1), (1,2)$ and we have already killed the kernel of $f^{\wedge \wedge}_D$, it follows that the kernel of the map in \eqref{eq:goal2} is contained in:

\begin{align}\label{eq:bigsource}
[D_{10}\otimes D_{01}] \oplus [D_{10}\otimes ((A\otimes B)/E')_{01}] \oplus [D_{01}\otimes  ((A\otimes B)/E')_{10}] \notag\\
\oplus [D_{10}\otimes ((A\otimes B)/E')_{11}] \oplus [D_{01}\otimes  ((A\otimes B)/E')_{11}] \oplus [D_{00}\otimes  ((A\otimes B)/E')_{11}] 
\end{align}
where the first factor is embedded in $S^2 D$ by $s\otimes t \mapsto \frac{1}{2} (s\otimes t + t\otimes s)$. The image of every nonzero element of $[D_{00}\otimes  ((A\otimes B)/E')_{11}]$ has a non-zero projection onto the subspace spanned by 
\[
\langle a^1_2\wedge a^3_3\otimes b^1_1\wedge b^2_2, a^2_2\wedge a^3_3\otimes b^1_1\wedge b^2_2, a^1_2\wedge a^3_3\otimes b^1_1\wedge b^2_3, a^2_2\wedge a^3_3\otimes b^1_1\wedge b^2_3\rangle
\]
while the image of the other direct summands of the source project trivially on this four-dimensional subspace of the target. The $52$-dimensional subspace of the source spanned by the following subspaces is contained in the kernel:
\begin{align}\label{eq:bigkernel}
&D_{01} \otimes (\langle a^3_1, a^3_2\rangle \otimes \langle b^2_2,  b^2_3, b^3_2, b^3_3\rangle \oplus \langle a^3_2\rangle \otimes \langle b^1_2, b^1_3\rangle \oplus \langle a^3_3\otimes b^1_1\rangle)\notag\\
&D_{10} \otimes (\langle a^1_1,a^1_2,a^1_3,a^2_1,a^2_2,a^2_3\rangle \otimes \langle b^2_1, b^3_1\rangle\oplus \langle a^3_3\otimes b^1_1\rangle)\notag\\
&\langle (T^1_1 \otimes T^2_3 + T^2_3 \otimes T^1_1) -2 T^1_1 \otimes a^3_3 \otimes b^3_2\rangle\\
&\langle (T^1_1 \otimes T^3_3 + T^3_3 \otimes T^1_1) -2 T^1_1 \otimes a^3_3 \otimes b^3_3\rangle\notag\\
&\langle (T^1_2 \otimes T^2_3 + T^2_3 \otimes T^1_2) -2 T^1_2 \otimes a^3_3 \otimes b^3_2\rangle \notag\\
&\langle (T^1_2 \otimes T^3_3 + T^3_3 \otimes T^1_2) -2 T^1_2 \otimes a^3_3 \otimes b^3_3\rangle.\notag
\end{align}
Taking the quotient of \eqref{eq:bigsource} by the subspace spanned by $[D_{00}\otimes  ((A\otimes B)/E')_{11}]$ and \eqref{eq:bigkernel}, establishing \eqref{eq:goal2} is equivalent to showing that
\[
[D_{01}\otimes \langle a^3_3\rangle \otimes \langle b^1_2,b^1_3,b^2_2, b^2_3, b^3_2, b^3_3\rangle ]  \oplus [D_{10}\otimes \langle a^1_2, a^1_3, a^2_2,a^2_3\rangle \otimes \langle b^1_1\rangle] \to (\Lambda^2 A \otimes \Lambda^2 B)/(\mathrm{im} f^{\wedge\wedge}_E)
\]
has kernel of dimension smaller than $5$. The image of this $20$-dimensional space is spanned by
\begin{align*}
&\langle a^3_3\wedge a^1_1, a^3_3\wedge a^2_1 \rangle \otimes \langle b^1_1\wedge b^1_2, b^1_1\wedge b^1_3, b^1_1\wedge b^2_2, b^1_1\wedge b^2_3, b^1_1\wedge b^3_2, b^1_1\wedge b^3_3\rangle\\
&\langle a^3_3\wedge a^1_2, a^3_3\wedge a^1_3, a^3_3\wedge a^2_2, a^3_3\wedge a^2_3\rangle \otimes \langle b^1_1\wedge b^3_2, b^1_1\wedge b^3_3\rangle.
\end{align*}
In particular the image is also $20$-dimensional so the kernel is trivial.

\begin{remark}\label{rmk:m3_16}
The original proof in \cite{chl19} involved enumerating triples of $\mathbb{B}$-fixed codimension $16$ subspaces of $S[X]_{110}, S[X]_{101}$ and $S[X]_{011}$. Using the $(220)$ square test, it is enough to look at a single degree $110$. Note that for an ideal $[I]$ in $\Hilb_{S[X]}^{h_{16, X}}$
\[
(I^2)_{220} = I_{110}\cdot I_{110} + I_{200}\cdot I_{020}.
\]
In our case already the first summand is too large in all $8$ candidate $110$ spaces. 
\end{remark}

\section{Proof of the general version of the square test}
In the main part of the paper we used border apolarity for three factor Segre varieties. Here we need a more general setup of Segre-Veronese varieties.
Let $X = \PP^{n_1}\times \cdots \times \PP^{n_k}$ with $\mathbb{Z}^k$-graded homogeneous coordinate ring $S[X]=\CC[\alpha_{1,0},\ldots, \alpha_{k,n_k}]$.

Let $\LL$ be a very ample line bundle on $X$. Let $\sigma_r(X,\LL)$ denote the $r$-th secant variety of $X \subseteq \PP(H^0(X, \LL)^\vee)$, i.e., the closure of the set of all points of $\PP(H^0(X, \LL)^\vee)$ of $X$-border rank at most $r$.
We say that $(X,\LL)$ is \emph{not-$r$-secant defective} if $\sigma_r(X, \LL)$ has the expected dimension, i.e., 
\[
\dim \sigma_r(X, \LL) = \min \{h^0(X, \LL)-1, r(\dim X+1)-1\}.
\]
The pair $(X,\LL)$ is $r$\emph{-secant defective} if the dimension of $\sigma_r(X,\LL)$ is smaller than the expected dimension. Given a point $p\in X$ with ideal sheaf $\mathcal{I}_p$ by a double point supported at $p$ we mean the subscheme of $X$ defined by $(\mathcal{I}_p)^2$.
We use the following well-known consequence of Terracini's lemma \cite{Ter11} whose proof can be found for example in \cite{CGG05}.
 \begin{proposition}\label{lem:TERRACINI}
     Let $X = \PP^{n_1}\times \cdots \times \PP^{n_k}$, let $r, d_1,\ldots, d_k$ be positive integers and let $Z$ be a general union of $r$ double points on $X$, then 
     \[
     \dim \sigma_r(X, \OO_X(d_1,d_2,\ldots, d_k)) = H_{S[X]/I_Z}(d_1,\ldots, d_k)-1.
     \]
 \end{proposition}
 
We can now prove our second test of the topological condition.

\begin{proposition}\label{prop:st_general}
Let $Y=\mathbb{P}^{n_1}\times \cdots \times \mathbb{P}^{n_j}$ and $X = Y \times \mathbb{P}^{n_{j+1}}\times \cdots \times \mathbb{P}^{n_k}$. Let $r, d_1, d_2,\ldots, d_j$ be positive integers such that 
$r(\dim Y+1) \leq \dim_\CC S[X]_{(d_1,\ldots, d_j, 0, \ldots, 0)}$ and $(Y,\OO_Y(d_1,\ldots, d_j))$ is not-$r$-secant defective. If $[I]\in \Slip{r}{X}$, then for all integers $e_1\geq d_1, \ldots, e_j\geq d_j$
\begin{equation}\label{ineq:st}
\HF{S[X]/I^2}{e_1,\ldots, e_j,0,\ldots, 0} \geq r(\dim Y+1).    
\end{equation}

\end{proposition} 
\begin{proof}
Let $S[Y]=\CC[\alpha_{1,0},\ldots, \alpha_{j,n_j}] \subseteq S[X]$ be the Cox ring of $Y$. It follows from \cite[Theorem~1.2]{Man23} that if $[I]\in \Slip{r}{X}$, then $[I\cap S[Y]]\in \Slip{r}{Y}$. Therefore, we may assume that $j=k$, i.e., that $Y=X$.

First we claim that $\HF{S[X]/I^2}{e_1,\ldots, e_k}$ is upper semicontinuous on the set of $\CC$-points of $\Hilb_{S[X]}^{h_{r,X}}$. For $k=1$ this is proved in \cite[Proof of Theorem~1.1]{Man22}.
The proof in the general case is similar. Therefore, in order to prove the proposition, it is sufficient to show that $\HF{S[X]/I^2}{e_1,\ldots, e_k} \geq r(\dim X+1)$ for $[I]\in \Hilb_{S[X]}^{h_{r,X}}$ that is an ideal of $r$ points in general position in $X$. Let $R = \{p_1,\ldots,p_r\}$ be the subscheme of $X$ defined by $I$ and let $Z$ be the union of $r$ double points with the same support as $R$.
We have 
\begin{equation}\label{eq:squarevssquare}
    I^2 = (I_{p_1}\cap \cdots \cap I_{p_r})^2\subseteq I_{p_1}^2\cap \cdots \cap I_{p_r}^2 = I_Z.
\end{equation}
Therefore, $\HF{S[X]/I^2}{e_1,\ldots, e_k}  \ge \HF{S[X]/I_Z}{e_1,\ldots, e_k}$. Since $(X,\OO_X(d_1,\ldots, d_k))$ is not-$r$-secant defective and $r(\dim X+1) \le H^0(X, \OO_X(d_1,\ldots,d_k))$
it follows by a standard argument with Castelnuovo exact sequence that $(X,\OO_X(e_1,\ldots, e_k))$ is also not-$r$-secant defective. Therefore, $\HF{S[X]/I_Z}{e_1,\ldots, e_k} = r(\dim X+1)$ by Proposition~\ref{lem:TERRACINI}. 
\end{proof}

We now specialize to the case $X=\PP^n\times \PP^n\times \PP^n$.

\begin{corollary}\label{c:211test}
    If $X=\PP^n\times \PP^n\times\PP^n$ for $n\geq 3$ and $r$ is a positive integer such that $(3n+1)r \leq (n+1)^3$, then 
    \begin{equation}
    \HF{S[X]/I^2}{e_1,e_2,e_3} \geq (3n+1)r    
    \end{equation}
	for all $e_1,e_2,e_3\ge 1$ and for all $[I]\in \Slip{r}{X}$.
\end{corollary}
\begin{proof}
Apply Proposition~\ref{prop:st_general} with $j=k=3$ and $d_1 = d_2 = d_3 = 1$. The fact that $(X, \OO_X(1,1,1))$ is not-$r$-secant defective is established in \cite[Theorem~4.4]{Lic85}.
\end{proof}

\begin{corollary}\label{c:220test}
    If $X=\PP^n\times \PP^n\times \PP^n$ and $r$ is a positive integer such that $(2n+1)r \leq \binom{n+2}{2}(n+1)$, then 
    \begin{equation}
    \HF{S[X]/I^2}{e_1,e_2,0} \geq (2n+1)r    
    \end{equation}
    for all $e_1,e_2\ge 1$ with $e_1 + e_2 \ge 3$ and for all $[I]\in \Slip{r}{X}.$
\end{corollary}
\begin{proof}
    Apply Proposition~\ref{prop:st_general} with $j=2$, $k=3$ and $(d_1,d_2) = (1,2)$ if $e_1\leq e_2$ and $(d_1,d_2)=(2,1)$ in the other case. The fact that $(Y, \OO_Y(d_1,d_2))$ is not-$r$-secant defective is established in \cite[Theorem~1.2]{Abo10}.
\end{proof}

\begin{remark}
    \hfill
    \begin{enumerate}
    \item The bound in Proposition~\ref{prop:st_general} is not sharp in general as the containment in \eqref{eq:squarevssquare} might be proper.
    For example, let $X=\PP^3\times \PP^3\times \PP^3$ and $r=7$. Let $J$ be the ideal of $7$ general double points of $X$ and $I$ be the ideal of $7$ general points of $X.$ It can be checked in \textit{Macaulay2} that $\HF{S[X]/J}{3,1,0} = 49$ but $\HF{S[X]/I^2}{3,1,0} = 53$.
    \item Observe that one does not need to know the full ideal to check if it passes a square test in a certain degree. In fact, sometimes it is enough to consider a single degree of the ideal. See for example Remark~\ref{rmk:m3_16}.
    \item Proposition~\ref{prop:st_general} and its corollaries can be generalized to the $r$-secant defective case. We explain how to do this for Proposition~\ref{prop:st_general} with $j=k$.
    Let $X = \PP^{n_1}\times \cdots \times \PP^{n_k}$ and $\mathbf{e} = (e_1,\ldots, e_k) \in \ZZ^{k}_{>0}$. Assume that $r(\dim X + 1) \le h^0(X, \OO_X(\mathbf{e}))$. Let $\delta_r(X, \OO_X(\mathbf{e})) = r(\dim X+1) -1  - \dim \sigma_r(X, \OO_X(\mathbf{e}))$ be the $r$-th defect of $(X, \OO_X(\mathbf{e}))$. If $[I]\in \Hilb_{S[X]}^{h_{r, X}}$ is in $\Slip{r}{X}$, then $H_{S[X]/I^2}(\mathbf{e}) \ge r(\dim X+1) - \delta_r(X, \OO_X(\mathbf{e}))$.
    \end{enumerate}
\end{remark}

\subsection{Revisiting the Veronese case}

By using the connection with secant defectivity we can improve on the bound on $d$ in \cite[Theorem~1.1]{Man22} in the special case of $k=2$ and $h = h_{r, \PP^n}$.

\begin{proposition}\label{prop:st_veronese}
    Let $n, r$ be positive integers with $r\geq n+1$ and $X= \PP^n$. Let $d_{n,r} = \min\{d \mid r(n+1) \leq \binom{n+d}{d}\}$.    
    If $[I]\in \Slip{r}{X}$, then $\HF{S[X]/I^2}{d} \geq (n+1)r$ for all 
    \begin{equation}\label{eq:inequality_veronese}
    d \geq \begin{cases} d_{n,r} + 1& \text{ if } (n,r) = (2,5) \text{ or } (n,r) =(4,14) \\ d_{n,r} & \text{ otherwise}\end{cases}.    
    \end{equation}
    The above bounds on $d$ are sharp.
    \end{proposition}
\begin{proof}
   By Proposition~\ref{prop:st_general} it is enough to consider the cases where $(X, \OO_{X}(d_{n,r}))$ is $r$-secant defective.
   By the result of Alexander and Hirschowitz \cite{AH95} there is a finite list of pairs $(n,r)$ that need to be checked. We first prove that \eqref{eq:inequality_veronese} holds.
   \begin{itemize}
       \item If $(n,r) = (2,5)$, then $d_{n,r} = 4$ and $(X, \OO_X(5))$ is not-$5$-secant defective.
       \item If $(n,r) = (3,9)$, then $d_{n,r} = 5$ and $(X, \OO_X(5))$ is not-$9$-secant defective.
       \item If $(n,r) = (4,7)$, then $d_{n,r} = 3$. The pair $(X, \OO_X(3))$ is $7$-secant defective but any ideal $I$ in $\Hilb_{S[X]}^{h_{r,X}}$ has $I_1 = 0$, so $\HF{S[X]/I^2}{3} = \dim_\CC S[X]_3 = 35 = (n+1)r$. For $d>d_{n,r}$ the pair $(X, \OO_X(d))$ is not-$7$-secant defective so Proposition~\ref{prop:st_general} applies.
       \item If $(n,r) = (4,14)$, then $d_{n,r} = 4$ and $(X, \OO_X(5))$ is not-$14$-secant defective.
   \end{itemize}
   By definition $d_{n,r}$ is the smallest integer $d$ for which $\HF{S[X]/K}{d} \ge r(n+1)$ is possible for any homogeneous ideal $K\subseteq S[X]$. Therefore, to check sharpness of \eqref{eq:inequality_veronese} it is enough to show that if $[I]\in \Slip{r}{X}$ for $(n,r) = (2,5)$ or $(n,r) = (4,14)$, then $\HF{S[X]/I^2}{d_{n,r}} < r(n+1)$. In both cases, $d_{n,r} = 4$ and $r(n+1) = \dim_\CC S[X]_4$. Furthermore, if $[I]$ is any ideal in $\Hilb_{S[X]}^{h_{r,X}}$, then $I_1 = 0$ and $\dim_\CC I_2 = 1$. It follows that $\HF{S[X]/I^2}{4} = \dim_\CC S[X]_4 - 1 = r(n+1) - 1 < r(n+1)$.
\end{proof}

\begin{remark}
Proposition~\ref{prop:st_general} improves the bounds on $d$ from \cite[Theorem~1.1]{Man22}. For example $d_{2,7} = 5 < 8$.
    \end{remark}

\section*{Acknowledgments}
I would like to thank Joseph M. Landsberg both for suggesting me tensors on which I could test my methods and for many useful comments on how to improve the presentation.
While working on this project I was partially supported by the NSF grant CCF-2203618. I would like to thank the Simons Institute for the Theory of Computing, where I worked on this project
during the fall 2025 program Complexity and Linear Algebra. My research was supported in part by Apple, Fall 2025.

	\bibliographystyle{abbrv}

\end{document}